\documentclass[reqno,10pt]{amsart}
\usepackage{amscd,amssymb,amsmath,color,url}
\usepackage{latexsym}
\usepackage[all]{xy}
\usepackage[colorlinks=true]{hyperref}
\usepackage{defs}

\def\tillam{{\wt{\lambda}}}
\def\nor{{\rm nor}}
\def\Ram{{\rm Ram}}
\def\slope{{\rm slope}}
\def\hyp{{\rm hyp}}
\def\tame{{\rm tame}}
\def\pord{{\rm pord}}
\def\logord{{\rm logord}}

\begin{document}

\author{Uri Brezner}
\address{Einstein Institute of Mathematics, The Hebrew University of Jerusalem, Giv'at Ram, Jerusalem, 91904, Israel}
\email{uri.brezner@mail.huji.ac.il}

\author{Michael Temkin}
\address{Einstein Institute of Mathematics, The Hebrew University of Jerusalem, Giv'at Ram, Jerusalem, 91904, Israel}
\email{michael.temkin@mail.huji.ac.il}

\date{\today}
\thanks{This work was supported by the Israel Science Foundation (grant No. 1159/15). The second author is grateful to Luc Illusie, Ilya Tyomkin and Yakov Varshasvsky for valuable discussions. We thank Andrew Obus for pointing out some typos and informing us about Henrio's work.}
\title{Lifting problem for minimally wild covers of Berkovich curves}

\begin{abstract}
This work continues the study of residually wild morphisms $f\:Y\to X$ of Berkovich curves initiated in \cite{CTT}. The different function $\delta_f$ introduced in \cite{CTT} is the primary discrete invariant of such covers. When $f$ is not residually tame, it provides a non-trivial enhancement of the classical invariant of $f$ consisting of morphisms of reductions $\tilf\:\tilY\to\tilX$ and metric skeletons $\Gamma_f\:\Gamma_Y\to\Gamma_X$. In this paper we interpret $\delta_f$ as the norm of the canonical trace section $\tau_f$ of the dualizing sheaf $\omega_f$, and introduce a finer reduction invariant $\tiltau_f$, which is (loosely speaking) a section of $\omega_\tilf^\rmlog$. Our main result generalizes a lifting theorem of Amini-Baker-Brugall\'e-Rabinoff from the case of residually tame morphism to the case of minimally residually wild morphisms. For such morphisms we describe all restrictions the datum $(\tilf,\Gamma_f,\delta|_{\Gamma_Y},\tiltau_f)$ satisfies, and prove that, conversely, any quadruple satisfying these restrictions can be lifted to a morphism of Berkovich curves.
\end{abstract}
\maketitle

\section{Introduction}

\subsection{Motivation}
This paper continues the study of residually wild morphisms $f\:Y\to X$ of Berkovich curves initiated in \cite{CTT}. Such morphisms were called ``topologically wild'' in \cite{CTT}, but we adopt the change of notation made in the sequel paper \cite{radialization}. The main new tool introduced in \cite{CTT} is the different function $\delta_f$ of $f$, and the main results of \cite{CTT} describe its properties, including the behaviour at points of types 1 and 2. A brief summary is given in appendix~\ref{diffunsec}. The results of \cite{CTT} sufficed to describe the structure of the residually wild locus of $f$ when $\deg(f)$ equals to the residual characteristic $p$, see \cite[Theorem~7.1.4]{CTT}, and were generalized in \cite{radialization} to an arbitrary degree. However, an analogy with the residually tame case indicated that one should expect existence of finer differential invariants related to algebraic geometry over $\tilk$. Indeed, it is noted in \cite[\S1.3]{ABBR} that the discrete invariant $\Gamma_f$ alone (see \S\ref{convsec}) is too crude to satisfy a lifting theorem. The set of natural discrete (or tropical) restrictions it satisfies is not complete, and there are non-trivial restrictions on $\Gamma_f$ arising from its relation to the algebraic geometry over $\tilk$. In particular, if one combines $\tilf$ and $\Gamma_f$ into a single object, a metrized curve complex of Amini-Baker (see \cite{Amini-Baker}) or a morphism of log curves (see \S\ref{logredsec}), then no essential information is lost under the reduction of residually tame morphism, as indicated by the lifting theorem \cite[Theorem~B]{ABBR}.

It was mentioned in \cite[Remark 7.6]{ABBR} that this theory is not satisfactory in the residually wild case, and the reason is now clear: there exist new non-trivial discrete invariants of covers, the different function $\delta_f$ and the profile function $\varphi_f$ (see \cite{radialization}). Note that $\delta_f$ is determined by $\varphi_f$, and the converse is true when $f$ is not too wild in the sense that the local multiplicities $n_y$ of $f$ are not divisible by $p^2$. It was natural to expect that these discrete (or tropical) invariants possess reduction refinements related to algebraic geometry over $\tilk$. Discovering such invariants and using them to extend the lifting theorem to a non-trivial residually wild case is the main motivation of this paper. We manage to completely solve the problem for the class of {\em minimally wild} covers (see the definition in \S\ref{convsec}).

\subsection{Methods and main results}

\subsubsection{The different and bivariant differential forms}
Probably, our main discovery is a relation between the different and the dualizing sheaves in the non-archimedean context. The most non-trivial property of the different proved in \cite[Theorem~4.5.4]{CTT} is a balancing condition on its slopes at a point $y\in Y$ of type 2: $$2g(y)-2-n_y(2g(x)-2)=\sum_{v\in C_y}(-\slope_v\delta_f+n_v-1),$$ where $x=f(y)$. We called it the local RH formula because it resembles the Riemann-Hurwitz formula for the induced morphism $\tilf_y\:C_y\to C_x$ of the reduction curves, and reduces to it when $\tilf_y$ is generically \'etale. In general, this formula was proved as follows. The homomorphism $\psi_{Y/X/k}\:f^*\Omega_{X,x}\to\Omega_{Y,y}$ shifts K\"ahler norms by $\delta_f(y)$, hence rescaling by an element $c\in k$ with $|c|=\delta_f(y)$, one obtains an isomorphism $\tilpsi\:\Omega_{\wHx/\tilk}\otimes_\wHx\wHy\toisom\Omega_{\wHy/\tilk}$, unique up to multiplication by an element $\tilc\in\tilk^\times$. Slopes of the different are expressed as the logarithmic orders of zeros and poles of the induced meromorphic map $\tilpsi\:\tilf_y^*\Omega_{C_x/\tilk}\to\Omega_{C_y/\tilk}$, hence the local RH formula boils down to the relation between the degrees of both sheaves and the logarithmic orders of $\tilpsi$ at the closed points of $C_y$.

Clearly, $\tilpsi$ can be viewed as a meromorphic section $\tiltau_f$ of $(\tilf_y^*\Omega_{C_x/\tilk})'\otimes\Omega_{C_y/\tilk}$, where $\calF'=\calHom(\calF,\calO_{C_x})$ denotes the dual of a coherent $\calO_{C_x}$-module $\calF$. We call such sections {\em bivariant differential forms} with respect to $\tilf_y$. Technically, they provide the correct setting for our results, and the first version of the paper used only such language. We are very grateful to Luc Illusie for the observation that $(\tilf_y^*\Omega_{C_x/\tilk})'\otimes\Omega_{C_y/\tilk}=\omega_{\tilf_y}$ is just the dualizing sheaf of $\tilf_y$. This allowed us to reinterpret the results in the following much more conceptual way. Consider the trace bivariant form $\tau_f\in\omega_f$ corresponding to the map $f^*\Omega_Y\to\Omega_X$; it is non-zero when $f$ is generically \'etale. The different $\delta_f$ is simply $\|\tau_f\|$ with respect to the natural norm on $\omega_f$ (see Theorem~\ref{deltatau}), so after an appropriate rescaling we can define the reduction $\tiltau_{f,y}$ as an element of the reduction sheaf $\tilomega_{f,C_y}$. A direct computation shows that the latter is $\omega^\rmlog_{\tilf_y}$ (see Theorem~\ref{redomega}), hence the slopes of the different at $y$ are equal to the logarithmic orders of $\tiltau_{f,y}$ at the points of $C_y$.

\subsubsection{Characterization of $\tiltau_{f,y}$}
It turns out that $\tiltau_{f,y}$ always satisfies certain restrictions, that we completely describe when $n_y=p$. Namely, $\tiltau_{f,y}$ is exact if $|p|<\delta_f(y)<1$ and mixes an exact and a logarithmic part when $\delta_f(y)=|p|$, see Theorem~\ref{degpth} and its converse Theorem~\ref{type2fields}. In particular, this explains strange combinatorial restrictions the slopes of $\delta_f$ always satisfy, see Remark~\ref{sloperem2}.

\begin{rem}
(i) We are grateful to Andrew Obus for pointing out that an essentially equivalent invariant was defined by Y. Henrio in his thesis, \cite{Henrio}, in the case when $f$ is a $\bfZ/p\bfZ$-Galois cover and the characteristic is mixed. In this case, $f$ is determined by a $\mu_p$-torsor, whose reduction on $C_y$ is defined by Henrio to be an exact meromorphic form $\phi=dt$ when $|p|<\delta_f(y)<1$, and a logarithmic differential form $\phi=\frac{dw}{w}$ when $\delta_f(y)=|p|$. Since $C_y\to C_x$ is the geometric Frobenius in this case, $\omega_{C_y/C_x}=\Omega_{C_y}^{\otimes (1-p)}$, and it is natural to expect that the two reduction invariants should be related by $\phi^{\otimes (1-p)}=\tiltau_{f,y}$. Indeed, Andrew Obus informed us that a computation he made confirms this.

(ii) Our definition applies in general without any relation to the Galois theory. In particular, it may happen that $f$ is minimally wild at $y$ and $\tiltau_{f,y}$ is not a $(p-1)$-th power. In this case, Henrio's invariant is not defined and the extension $\calH(y)/\calH(x)$ is not Galois.
\end{rem}

\subsubsection{$p$-enhanced reduction}
Assume that $f\:Y\to X$ is minimally wild. Given a compatible pair of semistable models one naturally obtains a log reduction map $\lam\:\tilY^\rmlog\to\tilX^\rmlog$ between the reductions with natural logarithmic structures (see \S\ref{logredsec} and \S\ref{logredsec2}). We provide $\lam$ with the additional datum $(\delta_f|_{V},\{\tiltau_{f,y}\}_{y\in V})$, where $V\subset Y$ is the set of points of type 2 corresponding to the generic points of $\tilY$. We call such a datum a {\em $p$-enhancement} of $\lam$ if it satisfies certain restrictions, and our results on $\delta_f$ and $\tiltau_{f,y}$ imply that $(\delta_f|_{V},\{\tiltau_{f,y}\}_{y\in V})$ is indeed a $p$-enhancement. This is Theorem~\ref{logredth} which generalizes \cite[Theorem~A]{ABBR} to covers which are minimally wild on the source.

\subsubsection{The lifting theorem}
Our main lifting result is Theorem~\ref{mainth} which asserts that any minimally wild morphism of log-curves $\lam\:C\to D$ provided with a $p$-enhancement $(\delta,\{\phi_y\})$ can be lifted to a minimally wild morphism $Y\to X$ of Berkovich curves. It is a generalization of \cite[Theorem~B]{ABBR} but it is substantially weaker in one aspect: we construct both $X$ and $Y$, rather than construct $f$ starting with a fixed $X$ whose log reduction is $D$.

The strategy of the proof is similar to that of \cite[Theorem~B]{ABBR}. We first construct separate liftings for irreducible components of $C$ and $D$ obtaining finite covers of star-shaped curves. Then we glue these covers along annular covers $h\:A'\to A$. If $h$ is residually tame then the cover is Kummer and the $A$-isomorphism class of $A'$ is determined by $\deg(h)$. This allows gluing in \cite[Theorem~B]{ABBR} even when $X$ is fixed. In the wild case, the set of $A$-isomorphism classes is huge, but if the degree is $p$ then we manage to classify all such covers up to automorphisms of both $A$ and $A'$. We prove in Theorem~\ref{binomth} that any $p$-cover $h$ is either Kummer or given by a binomial $x=y^p+cy^n$, and $h$ is determined by the restriction of $\delta_h$ onto the skeleton of $A'$. This allows gluing for minimally wild morphisms, at cost of a simultaneous gluing of both $X$ and $Y$ from star-shaped curves and annuli.

\subsubsection{General wild morphisms}
One might wonder if the results of this paper can be extended to arbitrary wild covers. We critically used the degree-$p$ assumptions at two places: in Theorem~\ref{degpth} when describing the reduction form $\tiltau_{f,y}$, and in Theorem~\ref{binomth} when classifying annular $p$-covers. We expect that using the splitting technique of \cite{radialization}, one can extend the description of $\tiltau_{f,y}$ and the notion of $p$-enhancement to the case of arbitrary wild covers. In particular, we expect the local lifting problem for points of type 2 to be treatable by our methods. The actual bottleneck is the classification of annular $p$-covers, and in order to push the lifting theorem further one will have to replace the gluing argument completely.

\subsection{The structure of the paper}
In Section \ref{onedim} we establish some facts about one-dimensional $k$-analytic fields $K$. Our main result there is Theorem~\ref{tameth} about representing elements of $K$ as sums of a $p$-th power and a $p$-orthogonal element. In Section~\ref{sec3} we develop the theory of reduction of bivariant forms. First, we have to metrize the sheaf $\omega_f$. We introduce $\calHom$-seminorms, slightly extending the toolkit of \cite{Temkintopforms}, and show that the so-defined K\"ahler norm on $\omega_f$ is $|k^\times|$-pm (abbreviation for piecewise $|k^\times|$-multiplicative). Main results about the reduction $\tiltau_{f,y}$ of the trace form are then proved in \S\ref{reddifsec}. In Section~\ref{annulisec} we classify isomorphism classes of annular covers of degree $p$. In Section~\ref{skelsec}, we introduce $p$-enhancements of morphisms of nice log curves and construct the reduction $p$-enhancements for covers which are minimally wild on the source, see Theorem~\ref{logredth}. Finally, our main lifting theorem is proved in Section~\ref{mainsec}. The paper has two appendices: we recall the basic properties of $\delta_f$ in appendix~\ref{diffunsec}, and discuss log reduction of non-archimedean curves in appendix~\ref{logredapp}.

\subsubsection{Conventions}\label{convsec}
Let us fix some notation. Throughout this paper $k$ denotes an algebraically closed complete real-valued field of positive residual characteristic $p$. By a nice $k$-analytic curve we mean a quasi-smooth connected compact separated strictly $k$-analytic curve $X$. Then $X^\hyp$ denotes the set of points of $X$ of type different from 1, and $X_G$ denotes the $G$-topological space of $X$. By an abuse of language, the associated topological space $|X_G|$ (\cite[Section~9]{Temkintopforms}) will be also denoted $X_G$.

Furthermore, $f\:Y\to X$ will always denote a finite generically \'etale morphism between nice curves. By $\Gamma_f\:\Gamma_Y\to\Gamma_X$ we mean any skeleton of $f$ in the sense of \cite[\S3.5.9]{CTT}. The multiplicity of $f$ at a point $y\in Y_G$ will be denoted $n_y$, see \cite[\S3.2.1, \S3.4.4]{CTT}. Finally, the different function of $f$ will be denoted $\delta_f$, \cite[Section~4]{CTT}.

If $f$ is not residually tame at $y\in Y$ then we say that $y$ is a {\em wild point} of $f$. If in addition, $n_y=p$ then we say that $y$ is {\em minimally wild}. Finally, we say that $f$ is {\em minimally wild on $Y$} if it only has minimally wild points, and $f$ is {\em minimally wild} if, in addition, any fiber $f^{-1}(x)$ has at most one wild point.

\section{One-dimensional $G$-analytic fields}\label{onedim}

\subsection{Trivial valuation}\label{trivsec}
Throughout Section \ref{trivsec}, $k$ is assumed to be trivially valued. In particular, $k=\tilk$ is of characteristic $p$. For expository reasons, we prefer to consider this case separately since it is very simple and illustrating.

\subsubsection{Differential orders on $k$-curves}
Assume that $X$ is a smooth $k$-curve. The orders $\ord_v$ of meromorphic functions at closed points $v$ are induced by the trivial $\calO_X$-lattice $\calO_X$ of the sheaf of meromorphic functions $\calM_X$. Similarly, the lattice $\Omega_{X/k}$ defines differential orders $\ord_v$ on meromorphic differential forms. We will also use the {\em logarithmic order} $\logord_v=\ord_v+1$ corresponding to the lattice $\Omega^\rmlog_{X/k}$. The latter is the huge quasi-coherent module obtained by twisting $\Omega_{X/k}$ by all closed points.

\subsubsection{Bivariant forms}
Assume that $f\:Y\to X$ is a finite morphism of smooth connected $k$-curves. Since $X,Y$ are smooth, the morphism is lci and the invertible sheaf $$\omega_f=\calHom_{f^{-1}\calO_X}(f^{-1}\Omega_X,\Omega_Y)=(f^*\Omega_X)'\otimes\Omega_Y$$ is the dualizing sheaf of $f$. Sections of $\omega_f$ will be called {\em bivariant differential forms} or simply bivariant forms. There is a canonical bivariant form $\tau_f$ corresponding to the map $\psi_{Y/X/k}\:f^*\Omega_X\to\Omega_Y$. In particular, $\tau_{f}$ vanishes if and only if $f$ is not generically \'etale.

For any differential form $\phi\in\Gamma(\Omega_X)$, we have that $\tau_f(\phi)=f^*(\phi)$, where we use the convention that the image of $\phi$ in $\Gamma(f^*\Omega_X)$ is denoted also by $\phi$, and $f^*\phi=\psi_{Y/X/k}(\phi)$ is the pullback of $\phi$ to a differential form on $Y$. In particular, $\tau_{f}(d_Xh)=d_Yf^*(h)$ for any $h\in\Gamma(\calO_X)$. With tensor notation, $\tau_{f}=\phi'\otimes f^*(\phi)$, where $\phi$ is a non-zero form on $X$ and $\phi'\in\Gamma(\Omega'_X)$ is its dual.

\begin{rem}
The element $\tau_f$ is the image of 1 under the trace map $$\calO_Y=f^*\calO_X\stackrel{t_f}{\to}f^!\calO_X=\omega_f.$$
This explains its importance for Grothendieck's duality. We will often call $\tau_f$ the {\em trace (bivariant) form} of $f$.
\end{rem}

\subsubsection{Orders of bivariant forms}
The generic stalk of $\omega_f$ is the one-dimensional $k(Y)$-vector space $\omega_{k(Y)/k(X)}=\Hom_{k(X)}(\Omega_{k(X)/k},\Omega_{k(Y)/k})$. Its elements will be called meromorphic bivariant forms. In addition to the $\calO_Y$-lattice $\omega_f$ of $\omega_{k(Y)/k(X)}$, we will also consider the natural logarithmic lattice $$\omega_f^\rmlog:=\calHom_{f^{-1}\calO_X}(f^{-1}\Omega^\rmlog_X,\Omega^\rmlog_Y).$$ Note that $\omega_f^\rmlog$ does not possess a tensor description because $\Omega^\rmlog_X$ is ''too large'' and its ``dual'' $\calHom_{\calO_X}(\Omega^\rmlog_X,\calO_X)$ vanishes.

For any closed point $v\in Y$ we obtain two induced orders on $\omega_{k(Y)/k(X)}$. Similarly to differential orders, they will be denoted $\ord_v$ and $\logord_v$. This will not lead to any confusion. Finally, it is easy to see that $\logord_v=\ord_v-e_v+1$, where $e_v$ is the ramification index at $v$.

\subsubsection{Laurent power series}
The orders at $v$ can be defined purely formally-locally via analogous constructions for fields $K=k((t))$. Note that $K$ is the only $k$-analytic field with non-trivial valuation (and its type is 3). Since the group of values is discrete, we will use the classical additive valuation instead of real valuations. It will be called the {\em order} and denoted $\ord\:|K|\to\bfZ\cup\{\infty\}$.

\subsubsection{$p$-order}\label{pordsec}
For any element $x\in K$ there exists $c\in K$ such that the order of $x-c^p$ is maximal. We say that $\pord(x):=\ord(x-c^p)$ is the {\em $p$-order} of $x$. It is infinite if and only if $x\in K^p$, and it lies in the set $\bfZ\setminus p\bfZ$ otherwise. Obviously, $\pord(x)\ge\ord(x)$ and the equality holds if and only if $\ord(x)\notin p\bfZ$. In fact, the $p$-order of $x=\sum_n c_nt^n$ is equal to the minimal $n\in\bfZ\setminus p\bfZ$ with a non-zero $c_n$.

\subsubsection{Order of differential forms}
The completed module of differentials $\hatOmega_{K/k}$ is a one-dimensional $K$-vector space generated by $dt$. The differential $\hatd_K\:K\to \hatOmega_{K/k}$ will be denoted by $d$ if no confusion is possible. There are two natural integral structures: the lattice of integral differential forms $\hatOmega_{\Kcirc/k}=\Kcirc dt$ generated over $\Kcirc$ by the elements $dx$ with $x\in\Kcirc$, and the lattice of logarithmic forms $\hatOmega_{\Kcirc/k}^\rmlog=\Kcirc\frac{dt}{t}$ generated over $\Kcirc$ by the logarithmic forms $\frac{dx}{x}$ with $x\in K^\times$. They define two orders on $\hatOmega_{K/k}$ that will be called the {\em differential order} and the {\em logarithmic differential order}, respectively. The word ``differential'' will usually be omitted. Clearly, $$\ord(f)=\ord(fdt)=\logord(fdt)-1.$$ We will write $\ord_K$ and $\logord_K$ when $K$ is not clear from the context.

\begin{rem}\label{ordrem}
(i) The differential orders are additive versions of the K\"ahler seminorm introduced in \cite{Temkintopforms}. The shift between the two is precisely the order of a uniformizer. In the non-discrete case the logarithmic seminorm coincides with the non-logarithmic one, see \cite{Temkintopforms}.

(ii) The formal order is compatible with the differential order on curves: if $X$ is a $k$-curve, $K=k(X)$ and $v\in X$, then $\ord_v(\phi)=\ord_{\wh{k(X)_v}}(\phi)$ for $\phi\in\Omega_K$. Indeed, $\Omega_{K/k}\into\hatOmega_{\hatK/k}$ is an isometry by \cite[Theorem~5.6.6]{Temkintopforms}.
\end{rem}

\subsubsection{Relation to the $p$-order}
The differential order provides a natural interpretation of the $p$-order:

\begin{lem}\label{pordlem}
Let $K=k((t))$. Then $\pord(x)=\logord(dx)$ for any element $x\in K$.
\end{lem}
\begin{proof}
Let $x=\sum_n c_nt^n$. We observed in \S\ref{pordsec} that $\pord(x)$ is the smallest $n\in\bfZ\setminus p\bfZ$ with $c_n\neq 0$. Clearly this number coincides with $\logord(dx)=\ord(\sum_n nc_nt^n)$.
\end{proof}

\subsubsection{The different}\label{diffsec}
Let $L/K$ be a finite separable extension. Then the different $\delta_{L/K}$ is the order of $\Ann(\Omega_{\Lcirc/\Kcirc})$ and the logarithmic different $\delta^\rmlog_{L/K}$ is the order of $\Ann(\Omega^\rmlog_{\Lcirc/\Kcirc})$. So, the following lemma reduces to unwinding the definitions.

\begin{lem}\label{diflem0}
Let $L/K$ be as above, then

(i) The different measures the difference between the differential order on $\hatOmega_{L/k}$ and the pullback of the differential order on $\hatOmega_{K/k}$, namely $$\delta_{L/K}=\ord_L(\omega_L)-e_{L/K}\ord_K(\omega)$$ for any non-zero $\omega\in\hatOmega_{K/k}$ and its image $\omega_L\in\hatOmega_{L/k}$.

(ii) The same relation holds for the logarithmic different: $$\delta^\rmlog_{L/K}=\logord_L(\omega_L)-e_{L/K}\logord_K(\omega).$$

(iii) The two differents are related by $\delta^\rmlog_{L/K}=\delta_{L/K}-e_{L/K}+1$.
\end{lem}

Choosing $\omega=dt$ we obtain the following specific way to compute the different.


\begin{cor}
Let $K=k((t))$ and $L/K$ be as above. Then $$\delta_{L/K}=\ord_L(dt)=\pord_L(t)-1.$$
\end{cor}

\subsubsection{Bivariant forms}
For a finite extension $L/K$ we define the one-dimensional $L$-vector space $$\omega_{L/K}:=\Hom_K(\hatOmega_{K/k},\hatOmega_{L/k})=(\hatOmega_{K/k})'\otimes_K\hatOmega_{L/k},$$ where $V'$ denotes the dual of $V$. Its elements will be called {\em bivariant differential forms}. There is a canonical element $\tau_{L/K}\in\omega_{L/K}$ satisfying $\tau_{L/K}(\hatd_Kx)=\hatd_Lx$ for any $x\in K$. It vanishes if and only if $L/K$ is inseparable. With tensor notation, $\tau_{L/K}=(\hatd_Kx)'\otimes\hatd_Lx$ for any $x$ with $\hatd_Kx\neq 0$.

\begin{rem}
It is easy to see that $\omega_{L/K}$ is the dualizing sheaf of $L/K$, whence the notation. The element $\tau_{L/K}$ is the trace element from duality theory, i.e. the image of $1$ under the trace map $L\to\omega_{L/K}$.
\end{rem}

\subsubsection{Orders of bivariant forms}
Similarly to the case of differential forms, we have natural orders $\ord_{L/K}$ and $\ord_{L/K}^\rmlog$ on $\omega_{L/K}$ defined by the lattices $$\omega_{\Lcirc/\Kcirc}:=\Hom_{\Kcirc}(\hatOmega_{\Kcirc/k},\hatOmega_{\Lcirc/k}),\ \ \ \omega_{\Lcirc/\Kcirc}^\rmlog:=\Hom_{\Kcirc}(\hatOmega^\rmlog_{\Kcirc/k},\hatOmega^\rmlog_{\Lcirc/k}).$$
It follows from Remark~\ref{ordrem}(ii) that these orders are compatible with analogous orders on $k$-curves. Clearly, both lattices contain $\tau_{L/K}$. Also, we can now reinterpret the different as follows:

\begin{lem}\label{taudiflem}
If $L/K$ is a finite extension, then $\delta_{L/K}=\ord_{L/K}(\tau_{L/K})$ and $\delta^\rmlog_{L/K}=\logord_{L/K}(\tau_{L/K})$. In particular, $\delta_{L/K}$ (resp. $\delta^\rmlog_{L/K}$) vanishes if and only if $\tau_{L/K}$ generates $\omega_{\Lcirc/\Kcirc}$ (resp. $\omega_{\Lcirc/\Kcirc}^\rmlog$).
\end{lem}
\begin{proof}
If $L/K$ is inseparable then $\tau_{L/K}=0$ and the differents are infinite. Assume that $L/K$ is separable. Choose $x\in K$ such that $\hatd_Kx\neq 0$. Then $$\ord_{L/K}(\tau_{L/K})=\ord_{L/K}\left((\hatd_Kx)'\otimes\hatd_Lx\right)=-e_{L/K}\ord_K(\hatd_Kx)+\ord_L(\hatd_Lx)$$ and similarly for the log orders. It remains to use Lemma~\ref{diflem0}.
\end{proof}


\subsection{Types of fields}
Similarly to \cite{CTT}, we will use some facts about one-dimensional analytic $k$-fields and their parameters, but this time more refined ones. Therefore it will be convenient to slightly revise some definitions in the mixed characteristic case.

\subsubsection{One-dimensional analytic fields}\label{onedimsec}
As in \cite[\S2.1.1]{CTT} or \cite[Section 6.2]{temst}, an analytic $k$-field $K$ is {\em one-dimensional} if it is finite over a subfield of the form $\wh{k(t)}$. Such fields are classified by the invariants $F_{K/k}=\trdeg_{\tilk}(\tilK)$ and $E_{K/k}=\dim_\bfQ(|K^\times|/|k^\times|\otimes_\bfZ\bfQ)$ into three types: type 2 has $F=1, E=0$, type 3 has $F=0,E=1$, and type 4 has $F=E=0$. If $X$ is a $k$-analytic curve and $x\in X$ is not Zariski closed then $\calH(x)$ is one-dimensional and its type is the type of $x$ in the classification of Berkovich.

\subsubsection{Fields of type 5}
Recall that in addition to the points of $X$, the topological space $X_G$ contains so-called points of type 5 whose completed residue fields are valued field of height 2, see \cite[Sections~3.1 and 3.4]{CTT}. This motivates the following definition. A valued $k$-field $K$ is called {\em one-dimensional $G$-analytic} if either it is a one-dimensional analytic field, or $K$ is of the following special form:

The valuation $|\ |$ on $K$ has values in $\bfR_{>0}^\times\times\bfZ\cup\{0\}$ ordered lexicographically, and denoting the first component by $|\ |_1\:K\to\bfR_{\ge 0}$, we have that $K_1=(K,|\ |_1)$ is a one-dimensional analytic $k$-field. In this case, we say that $K$ is of {\em type 5}. The second projection $\lam\:K^\times\to\bfZ$ is a multiplicative map but not a valuation. However, $-\lam$ induces a discrete valuation $\tilK_1^\times\onto\bfZ$ trivial on $\tilk$. It follows that $K_1$ is of type 2 and the valuation of $K$ is composed from the valuation of $K_1$ and a discrete valuation on $\tilK_1$ corresponding to a point of the smooth proper $\tilk$-model of $\tilK_1$.

\subsection{Parameters}

\subsubsection{$p$-seminorm}
For a real-valued $k$-field $K$ we define the {\em $p$-seminorm} by $|x|_p=\inf_{c\in K}|x-c^p|$. It heavily depends on $K$ and can drop in extensions, so we will use the full notation $|\ |_{K,p}$ when needed. This is a straightforward extension of the $p$-order to arbitrary real-valued fields.

\subsubsection{Best $p$-power approximations and $p$-orthogonality}
More generally, assume that $K$ is a valued field. If an element $x\in K$ is such that the infimum is achieved: $|x-b^p|=\min_{c\in K}|x-c^p|$ for some $b\in K$, then we set $|x|_p=|x-b^p|$ and say that $b^p$ is a {\em best $p$-power approximation} of $x$ (in $K$). In the particular case when $|x|_p=|x|$, we say that $x$ is {\em $p$-orthogonal} (in $K$). In general, one can define $|x|_p$ as a cut on the group of values $|K^\times|$ rather than an element of $|K^\times|$, but we will not need this. The only convention we will use for a general $x\in K$ is that $|x|_p\ge r\in|K^\times|$ means that $|x-c^p|\ge r$ for any $c\in K$.

\begin{lem}\label{paproxlem}
Assume $K$ is a valued field and $x\in K$ satisfies $|x|_p\ge|px|$. Then an element $c^p$ is a best $p$-power approximation if and only if $x-c^p$ is $p$-orthogonal.
\end{lem}
\begin{proof}
It suffices to consider the case when $|x-c^p|\le|x|$. If $c^p$ is not a best $p$-power approximation then $|x-b^p|<|x-c^p|$ for some $b\in K$. It is easy to see that $|(c-b)^p-c^p+b^p|\le|px|<|x-c^p|$. Hence $|x-c^p+(c-b)^p|\le|x-b^p|<|x-c^p|$, and we obtain that $x-c^p$ is not $p$-orthogonal. Conversely, if $x-c^p$ is not $p$-orthogonal, then $|x-c^p-a^p|<|x-c^p|$ for some $a\in K$. It follows easily that $|x-(a+c)^p|<|x-c^p|$, and hence $c^p$ is not a best $p$-power approximation.
\end{proof}

Now, we can prove our main result about existence of $p$-power approximations.

\begin{theor}\label{minth}
Let $K$ be a one-dimensional $G$-analytic $k$-field of type 2, 3 or 5 and let $x\in K$ be an element such that $|x|_p\ge|px|$. Then $x$ possesses a best $p$-power approximation $c^p\in K^p$.
\end{theor}
\begin{proof}
Assume that $K$ is of type 2 or 3 first. By \cite[Theorem~6.3.1(i)]{temst}, $K$ is unramified over a subfield of the form $\wh{k(y)}$, hence by \cite[Proposition~6.2.5]{temst} there exists an orthogonal Schauder basis $B$ of $K$ over $k$ of a very special form: $B=\{1\}\cup\{u^{p^n}|\ u\in U,n\in\bfN\}$, where $U\subset K$ is a subset such that any element of $\Span_k(U)$ is $p$-orthogonal. (The latter proposition is an important ingredient in proving that $K$ is stable.) Since $B=U\coprod B^p$ and $k=k^a$, we can represent $x$ as $\sum_{u\in U}a_uu+\sum_{b\in B}c^p_bb^p$. Set $t=\sum_{u\in U}a_uu$ and $c=\sum_{b\in B}c_bb$, and let us check that $c$ is as required.

Since $B$ is an orthogonal basis, we have that $|x|=\max(|t|,|c|^p)$ and $|x-t-c^p|\le|pc^p|\le|px|$. If $|t|=|px|$ then $|x-c^p|=|px|$, and hence $|x|_p=|px|$ and $c$ is as required. If $|t|>|px|$ then $|x-c^p-t|<|t|$ and since $t$ is $p$-orthogonal by the choice of $B$, $t+(x-c^p-t)=x-c^p$ is $p$-orthogonal too. By Lemma~\ref{paproxlem}, $c$ is as required.

Assume now that $K$ is of type 5, and so $|y|=(|y|_1,\lam(y))$, where $K_1=(K,|\ |_1)$ is the associated field of type 2. Let $C$ be the set of elements $x-c^p\in K$ such that $r=|x-c^p|_1$ is minimal. Then $C\neq\emptyset$ by the case of type 2, and by the discreteness of $\lam$ it suffices to show that the set $\lam(C)$ is bounded from below in $\bfZ$. If $r=|px|_1$, then $\lam(px)$ is such a bound.

It remains to show that the assumption that $r>|px|_1$ and $\lam(C)$ is unbounded from below leads to a contradiction. Choose $b\in k$ such that $|b|_1=r^{-1/p}$, then replacing $x$ by $x/b^p$ we can achieve that $|C|_1=1>|px|_1$. In this case, for any pair of elements $x-c^p$ and $x-b^p$ in $C$, the difference $b^p-c^p$ satisfies $|b^p-c^p-(b-c)^p|_1<1$. It follows that the image $\tilC\subset F:=\tilK_1$ of $C$ is a coset of $F^p$, that is, $\tilC=\tily+F^p$ for any $y\in C$. Let $\tillam$ be the discrete valuation $-\lam$ induces on $F$. Any $y\in C$ satisfies $|y|_1=1$ and hence $\lam(y)=-\tillam(\tily)$. We obtain that $\tillam(\tilC)$ is unbounded from above, hence the coset is trivial by \S\ref{pordsec}: $\tilC=F^p$. This immediately implies that $y\in C$ is not $p$-orthogonal in $K_1$, and we obtain the contradiction with Lemma~\ref{paproxlem}.
\end{proof}

\subsubsection{Pure and mixed parameters}
Let $K$ be a one-dimensional $G$-analytic $k$-field. By a {\em parameter} of $K$ we mean any element $t\in K$ such that $0\neq|t|_p\ge|pt|$. We say that the parameter is {\em pure} if $|t|_p>|pt|$ and {\em mixed} if $|t|_p=|pt|$. The second case can only occur in the mixed characteristic case.

\begin{rem}
A similar definition in \cite{CTT} only requires that $K/\wh{k(t)}$ is finite and separable. This condition is equivalent to $t\notin k$ if $\cha(k)=0$ and to $t\notin K^p$ if $\cha(k)=p$. Thus, the new definition can be viewed as a refinement of the old one, which is only essential in the mixed characteristic case.
\end{rem}

\subsubsection{Tame and monomial parameters}
Given a parameter $t$ we use the same definitions as in \cite{CTT}: the {\em radius} of $t$ is $r_t=\inf_{c\in k}|t-c|$, we say that $t\in K$ is {\em monomial} if $|t|=r_t$, and $t$ is {\em tame} if $K/\wh{k(t)}$ is tame.

\begin{lem}\label{tameparamlem}
Assume that $K$ is a one-dimensional $G$-analytic $k$-field and $x\in K$ is an element.

(i) If $x$ is a tame parameter then $|x|_p=r_x$.

(ii) If $K$ is of type 2, 3 or 5, then $x$ is a tame monomial parameter if and only if it is $p$-orthogonal.
\end{lem}
\begin{proof}
(ii) Fields of type 2, 3 and 5 are stable, hence the extension $K/\wh{k(x)}$ is defectless. If $K$ is of type 2, then multiplying $x$ by an element of $k$ we can assume that $|x|=1$. Then $x$ is a tame monomial parameter if and only if $\tilK$ is separable over $\wt{k(x)}=\tilk(\tilx)$, which happens if and only if $\tilx\notin\tilK^p$. It is easy to see that the latter happens if and only if $x$ is $p$-orthogonal.

Assume that $K=(K,|\ |)$ is of types 3 or 5, and hence $\tilK=\tilk$. If $x$ is monomial then $|k(x)^\times|$ is generated over $|k^\times|$ by $|x|$, and it follows that $x$ is a tame monomial parameter if and only if $|x|\notin|K^\times|^p$. On the other hand, since $\tilK=\tilk$ is perfect, it is easy to see that $x$ is $p$-orthogonal if and only if $|x|\notin|K^\times|^p$.

(i) For types 2, 3 and 5, subtracting from $x$ an element $c\in k$ we can make it monomial, and then the assertion reduces to the direct implication in (ii). So let $K$ be of type 4. Assume that $|x|_p\neq r_x$. Subtracting from $x$ an element $c\in k$ we can assume that $|px|<r_x$. Since $k=k^a$ we have that $|x|_p\le r_x$, and hence $|x-y^p|<r_x$ for some $y\in K$. Then $[K:\wh{k(y^p)}]>1$ by \cite[Lemma~6.2.8]{temst} and $[K:\wh{k(y^p)}]=[K\:\wh{k(x)}]$ by \cite[Lemma~6.3.3]{temst}. Thus $K/\wh{k(x)}$ is non-trivial, and being immediate it is necessarily wildly ramified. This proves that $x$ is not a tame parameter when $|x|_p\neq r_x$.
\end{proof}

\subsubsection{Taming a parameter}
Here is the main result on arbitrary parameters that connects them to tame ones.

\begin{theor}\label{tameth}
Assume that $K$ is a one-dimensional $G$-analytic $k$-field and $x\in K$ is a parameter. Then there exists a decomposition $x=c^p+t$, where $t\in K$ is a tame parameter and $c\in K$ is an element such that $|c^p|\le|x|$. If $K$ is not of type 4, then one can achieve in addition that $t$ is monomial.
\end{theor}
\begin{proof}
For types 2, 3 and 5, by Theorem~\ref{minth} we can find a best $p$-power approximation $c^p$ of $x$. Then $t=x-c^p$ is $p$-orthogonal by Lemma~\ref{paproxlem}, and hence it is a tame monomial parameter by Lemma~\ref{tameparamlem}. For the type 4 case this follows from \cite[Proposition~6.2.4]{temst} applied to the special coset $$x+S_{0,|x|_p}(K):=x+\{c^p+d|\ c,d\in K, |pc^p|<|x|_p,|d|<|x|_p\}.$$
\end{proof}

\subsubsection{Tame term}
Given a $k$-field $K$ of type 2, 3 or 5 and a parameter $x\in K$, consider the decomposition $x=c^p+t$ as in Theorem~\ref{tameth}. We call $t$ a {\em tame term} of $x$. Note that $|t|=|x|_p$ and $t$ is unique up to adding elements $d+b^p\in K$ such that $|d|<|b^p|=|t|$. In addition, for any element $x\in K$, which is not a parameter, we define its tame term to be equal to 0.

\subsubsection{Tame reduction}
When $K$ is of type 2 we will also want to consider an informative reduction data associated to the tame term $t$ of $x$. It will be an element $\tilx_\tame$ that we will call {\em the tame reduction of $x$}. If $x$ is not a parameter we set $\tilx_\tame=0$. Otherwise, choose any $a\in k$ with $|a|=|x|_p=|t|$, consider the element $t'=t/a$, and set $\tilx_\tame=\tilt'$. Up to adding elements $b^p$ with $|b|\le 1$ and multiplication by elements $d\in k$ with $|d|=1$, the element $t/a$ depends only on $x$, hence the tame reduction is well defined up to adding elements from $\tilK^p$ and multiplying by elements of $\tilk^\times$.

In the special case when $|t|_p=|p|$ we have a canonical choice $a=p$. The corresponding tame reduction will be called {\em canonical tame reduction}. It is well defined up to adding elements from $\tilK^p$. In fact, we will only need this in the special case when $x$ is a mixed parameter and $|x|=1$.

\section{Differentials and reduction}\label{sec3}

\subsection{The case of analytic fields}\label{anfieldssec}
In Section \ref{anfieldssec} $K$ is a one-dimensional analytic $k$-field. The case of type 5 could be dealt with similarly, but would require more work since some foundations were not developed in \cite{Temkintopforms} and \cite{CTT}.

\subsubsection{Completed modules of differentials}
Following \cite[\S4.1.1]{Temkintopforms}, we provide the vector space $\Omega_{K/k}$ with a natural K\"ahler seminorm $\|\ \|_{K/k,\Omega}$ and denote the completion $\hatOmega_{K/k}$. For shortness, the differential $\hatd_{K/k}\:K\to\hatOmega_{K/k}$ will be usually denoted simply by $d$ and the seminorm will be denoted by $\|\ ||$. Note that $\|dc\|\le r_c$ for any $c\in K$. By \cite[Theorem~2.3.2]{CTT}, $\hatOmega_{K/k}$ is one-dimensional, and if $t\in K$ is a tame parameter then $\|dt\|=r_t$. In particular, $dt$ is a basis of $\hatOmega_{K/k}$.

\subsubsection{Reduction}\label{redsec}
We will be mainly interested in the case when $K$ is of type 2. Then there exists a tame monomial parameter $t$ with $|t|=1$ and hence $\|dt\|=1$. In particular, $dt$ generates the unit ball $\hatOmega^\di_{K/k}$ of $\hatOmega_{K/k}$ and hence the natural map $h\:\hatOmega_{\Kcirc/\kcirc}\to\hatOmega^\di_{K/k}$ is an isomorphism. Using the unit ball $\hatOmega_{\Kcirc/\kcirc}$, we define the {\em reduction} of $\hatOmega_{K/k}$ to be the $\tilK$-vector space $\hatOmega^\di_{K/k}\otimes_{\kcirc}\tilk$. Since $h$ is an isomorphism this space is canonically isomorphic to $\Omega_{\tilK/\tilk}$ and we obtain the reduction map $\hatOmega^\di_{K/k}\to\Omega_{\tilK/\tilk}$ that will be denoted $\omega\mapsto\tilomega$. For any $x\in\Kcirc$, the reduction of $\hatd_{K/k}(x)$ is $d_{\tilK/\tilk}(\tilx)$.

\subsubsection{$p$-seminorm versus K\"ahler seminorm}
The K\"ahler seminorm can be also used to compute the $p$-seminorm in the most important cases. This result will not be used, but we add it for the sake of completeness.

\begin{theor}\label{normsth}
Assume that $K$ is a one-dimensional analytic $k$-field and $x\in K$ is a parameter. Then $|x|_p=\|\hatd_{K/k}x\|$.
\end{theor}
\begin{proof}
Find a presentation $x=c^p+t$ as in Theorem~\ref{tameth}, then $dx=dt+pc^{p-1}dc$ and $\|dt\|=r_t$. By Lemma~\ref{tameparamlem} we have that $r_t=|t|_p=|x|_p$. In addition, $|px|\ge|pc^p|\ge\|pc^{p-1}dc\|$ because $\|dc\|\le|c|$. If $|x|_p>|px|$ then $\|dt\|=|x|_p>\|pc^{p-1}dc\|$ and hence $\|dx\|=\|dt\|=|x|_p$. Assume now that $|x|_p=|px|$, and let us study different types separately.

If $K$ is of type 4 then $\|dc\|\le r_c<|c|$ hence we still have that $\|dt\|>\|pc^{p-1}dc\|$ and the argument works. If $K$ is of types 2 or 3 then by Theorem~\ref{tameth} we can assume that $t$ is tame and monomial, and hence $p$-orthogonal by Lemma~\ref{tameparamlem}. If $K$ is of type 3 then $|t|\notin|K^\times|^p$. However, $|p|\in|K^\times|^p$, and $|x|\in|K^\times|^p$ because $|x-c^p|<|x|$. This contradicts that $|t|=|x|_p=|px|$. Finally, assume that $K$ is of type 2. Then $|x|_p=|t|\in|k^\times|$, hence multiplying $x$ by an appropriate element of $k$ we can achieve that $|x|_p=|t|=1$. Since $t$ is $p$-orthogonal, $\tilt\notin|\tilK^p|$ and hence $d\tilt\neq 0$. Setting $b=cp^{-1/p}$ we have that $dx=dt+b^{p-1}db$, and hence the reduction is $d\tilx=d\tilt+\tilb^{p-1}d\tilb$. Zero is the only exact differential form of the form $\tilb^{p-1}d\tilb$, hence $d\tilx\neq 0$ and we obtain that $\|dx\|=1=|x|_p$, as required.
\end{proof}

\subsubsection{Bivariant differentials}\label{bivardifsec}
If $L/K$ is a finite extension we consider the vector space of bivariant differentials $\omega_{L/K}:=(\hatOmega_K)'\otimes_K\hatOmega_L$ and provide it with the norm induced from the K\"ahler norms on $\hatOmega_K$ and $\hatOmega_L$. The unit ball $\omega^\di_{L/K}$ of this norm coincides with $\omega_{\Lcirc/\Kcirc}:=(\hatOmega_{\Kcirc/\kcirc})'\otimes_{\Kcirc}\hatOmega_{\Lcirc/\kcirc}$. Again, when $K$ is of type 2 it is easy to compute the reduction:
$$\tilomega_{L/K}:=\omega^\di_{L/K}\otimes_{\kcirc}\tilk=(\Omega_{\tilK/\tilk})'\otimes_\tilK\Omega_{\tilL/\tilk}=\omega_{\tilL/\tilK}.$$

\subsubsection{The canonical bivariant form}
Finally, we have a canonical element $\tau_{L/K}$ corresponding to the map $\hatOmega_K\otimes_KL\to\hatOmega_L$.

\begin{theor}\label{taufields}
Let $L/K$ be a finite extension of one-dimensional analytic $k$-fields. Then,

(i) $\|\tau_{L/K}\|=\delta_{L/K}$.

(ii) If $K$ is of type 2, $L/K$ is separable, and $c\in k$ satisfies $|c^{-1}|=\delta_{L/K}$, then $\wt{c\tau_{L/K}}=(\tilphi)'\otimes\tilpsi\neq 0$, where $\phi\in\hatOmega_{K/k}$ satisfies $\|\phi\|=1$, and $\psi\in\hatOmega_L$ is the image of $c\phi$.
\end{theor}
\begin{proof}
Both sides in (i) vanish if the extension is inseparable, so assume that $L/K$ is separable. For any non-zero $\phi\in\hatOmega_{K/k}$ with image $\psi_1\in\hatOmega_{L/k}$, the ratio $\|\psi_1\|/\|\phi\|$ equals the same number $r$ because the spaces are one-dimensional. Clearly, $\|\tau_{L/K}\|=r$. On the other hand, $r=\delta_{L/K}$ by \cite[Theorem~2.4.4]{CTT}, thus proving (i). In addition, in (ii) we obtain that $c\tau_{L/K}=\phi'\otimes\psi$ and $\|\phi\|=\|\psi\|=1$, which implies the assertion.
\end{proof}

\subsection{K\"ahler seminorms on curves}\label{Kahsec}
Following \cite[\S6.1.1]{Temkintopforms} we provide $\Omega_X$ with the K\"ahler seminorm $\|\ \|=\|\ \|_{X/k,\Omega}$ and denote its unit ball by $\Omega^\di_X$, see also \cite[Remark~6.1.6]{Temkintopforms}.

\subsubsection{Pm functions}
Recall that a $k$-analytic curve $X$ possesses a canonical metric structure, e.g. see \cite[\S3.6.1]{CTT}. As in \cite[\S3.6.3]{CTT}, we say that a real-valued function on a subset $S\subseteq X$ is {\em piecewise $|k^\times|$-monomial} or {\em $|k^\times|$-pm} if its restrictions onto intervals in $S$ are so. For example, for any function $f$ its norm $|f|$ is a $|k^\times|$-pm function.

\subsubsection{$\calHom$-seminorms}
Seminorms on sheaves of $\calO_X$-modules and basic operations on them were defined in \cite[Section 3]{Temkintopforms}. We will also need a notion of $\Hom$-seminorm, but one should be careful with boundedness. If $M,N$ are seminormed modules over a seminormed ring $A$ then $L=\Hom_A(M,N)$ is provided with a {\em $\Hom$-quasi-norm} $|\phi|_L=\sup(|\phi(m)|_N/|m|_M)$ where $m$ runs over elements $m\in M$ with $|m|\neq 0$. The above supremum is finite if and only if $\phi$ is bounded, so $|\ |_L$ defines a {\em $\Hom$-seminorm} on the module $\Hom^b_A(M,N)$ of bounded homomorphisms.

Assume now that $\calF$ and $\calG$ are seminormed $\calO_{X_G}$-modules. We define the sheaf of bounded homomorphisms $\calHom^b_{\calO_{X_G}}(\calF,\calG)$ by sheafifying the presheaf that maps $U$ to $\Hom^b_{\calO_{X_G}(U)}(\calF(U),\calG(U))$. The sheafification is not needed if $X$ is compact. This sheaf is provided with the {\em $\calHom$-seminorm} obtained by sheafifying the $\Hom$-seminorms on the modules of sections.

\subsubsection{Pm seminorms}
Let $\calF$ be a sheaf on $X_G$ and $\|\ \|$ a seminorm on $\calF$. Any section $s\in\calF(U)$ determines the function $\|s\|\:U\to\bfR_{\ge 0}$ sending $x\in U$ to $\|s\|_x$. In general, these functions are only upper semicontinuous, but in many cases they also satisfy additional properties. We say that the seminorm $\|\ \|$ is {\em $|k^\times|$-pm} if $\|s\|$ is $|k^\times|$-pm for any section $s\neq 0$.

\begin{rem}
(i) It is an interesting problem to develop a theory of $|k^\times|$-pm seminorms for arbitrary $k$-analytic spaces using a notion of pm subspaces generalizing pl subspaces of \cite[Section~7.2]{Temkintopforms}. We do not pursue this here, and will only prove a few very particular results we will need.

(ii) K\"ahler seminorms and seminorms obtained from them should certainly be $|k^\times|$-pm, in particular, see \cite[Theorem~8.1.6]{Temkintopforms}. We will see that for curves this is indeed so, and the arguments are much simpler than in \cite[Theorem~8.1.6]{Temkintopforms}.
\end{rem}

\begin{lem}\label{sempmlem}
Let $X$ be a nice $k$-analytic curve, $\calF$ an invertible $\calO_{X_G}$-module, and $\|\ \|$ a seminorm on $\calF$. Then $\|\ \|$ is $|k^\times|$-pm if and only if for any interval $I\subset X$ with the induced $|k^\times|$-pm structure, the restriction of the unit ball $\calF^\di|_I$ is an invertible $\calOcirc_{X_G}|_I$-module.
\end{lem}
\begin{proof}
Working locally on $X$ we can assume that $s\in\calF$ is a generator. If $\|\ \|$ is $|k^\times|$-pm then the restriction of $\|s\|$ onto $I$ is $|k^\times|$-pm, hence there exists an admissible covering $X=\cup X_i$ with functions $f_i\in\calO_{X_G}(X_i)$ such that $\|s\|=|f_i|$ on $I_i=X_i\cap I$. Therefore $\|f_i^{-1}s\|=1$ on $I_i$ and it follows that $f_i^{-1}s$ is an $\calOcirc_{X_G}|_{I_i}$-generator of $\calF^\di|_{I_i}$. Thus, $\calF^\di|_I$ is invertible. The inverse implication is proved similarly, and we omit the details.
\end{proof}

The lemma implies that the K\"ahler seminorm on $\Omega_X$ is $|k^\times|$-pm because on any annulus $A$ with skeleton $l$ and coordinate $t$, the sheaf $\Omega^\di_A|_l$ is generated by $\frac{dt}{t}$, e.g. see \cite[Theorem~4.3.3(ii)]{CTT}.

\begin{cor}\label{sempmcor}
Basic operations on seminorms on invertible sheaves, including tensor products, duals, and pullbacks, preserve the property of being a $|k^\times|$-pm seminorm and respect unit balls. For example, given $|k^\times|$-pm seminorms on $\calF$ and $\calG$, the $\calHom$-seminorm on $\calH=\calHom^b_{\calO_{X_G}}(\calF,\calG)$ is $|k^\times|$-pm, and $\calH^\di=\calHom_{\calOcirc_{X_G}}(\calF^\di,\calG^\di)$.
\end{cor}

\subsubsection{The dualizing sheaf}
We do not develop any duality theory. In an ad hoc manner, we just call $$\omega_{f}:=\calHom_{\calO_{Y_G}}(f^*\Omega_{X_G},\Omega_{Y_G})=(f^*\Omega_{X_G})'\otimes_{\calO_{Y_G}}\Omega_{Y_G}$$ the dualizing sheaf of a finite morphism $f\:Y\to X$ of nice $k$-analytic curves. By $\tau_f\in\Gamma(\omega_{f})$ we denote the global bivariant form corresponding to the natural map $\psi_{Y_G/X_G/k}\:f^*\Omega_{X_G}\to\Omega_{Y_G}$. The K\"ahler seminorms on $\Omega_{X_G}$ and $\Omega_{Y_G}$ induce a $\calHom$-quasi-norm $\|\ \|_\omega$ on $\omega_{f}$.

\begin{lem}\label{boundlem}
If $f$ is generically \'etale then $\omega_{f}=\calHom^b_{\calO_{Y_G}}(f^*\Omega_{X_G},\Omega_{Y_G})$. In particular, $\|\ \|_\omega$ is a seminorm on $\omega_f$.
\end{lem}
\begin{proof}
By the assumption on $f$, we have that $\tau\neq 0$ and hence $\tau$ spans the invertible module $\omega_{f}$. It remains to show that $\tau$ is bounded. In fact, one even has that $\|\tau_f\|_\omega\le 1$ because the map $\psi_{Y_G/X_G/k}$ is contracting.
\end{proof}

The following remark will not be used, so we only indicate its justification.

\begin{rem}\label{boundrem}
One can show that $\tau$ always spans the module of bounded bivariant forms. So, if $f$ is not generically \'etale, then $\calHom^b_{\calO_{Y_G}}(f^*\Omega_{X_G},\Omega_{Y_G})=0$. Let us illustrate this with the simplest example when $\cha(k)=p>0$, $Y=\calM(k\{t\})$ and $X=\calM(k\{t^p\})$. It follows from \cite[\S6.2.1]{Temkintopforms} that the module $\Gamma(\Omega^\di_Y)$ consists of all elements $hd_Yt$ with $h\in k\{t\}$ satisfying $|h|_y\le r(y)^{-1}$ for $y\in Y$. In particular, $h$ may have at most simple poles at the closed points. By the same computation, $\Gamma(f^*\Omega^\di_Y)$ consists of all elements $hd_Xt^p$ with $h\in k\{t\}$ satisfying $|h|_y\le r(y)^{-p}$. Clearly, the second module admits no non-zero maps to the first one. Loosely speaking, the sheaf $\Omega_Y^\di$ is huge, but $f^*\Omega^\di_{X_G}$ is even much larger and cannot be embedded in it.
\end{rem}

Combining Corollary~\ref{sempmcor} and Lemma~\ref{boundlem} we obtain

\begin{cor}\label{omegaball}
Let $f$ be a generically \'etale morphism of nice $k$-analytic curves. Then the seminorm $\|\ \|_\omega$ on $\omega_{f}$ is $|k^\times|$-pm and its unit ball can be expressed as $$\omega_{f}^\di=\calHom_{\calOcirc_{Y_G}}(f^*(\Omega^\di_{X_G}),\Omega^\di_{Y_G}).$$
\end{cor}

\begin{rem}
This time we do not have a tensor description of the unit ball. The reason is that the bounded dual of $\Omega_{X_G}$ vanishes. Again, this happens because $\Omega^\di_{X_G}$ is much larger than $\calOcirc_{X_G}$ and cannot be embedded into it.
\end{rem}

\subsubsection{Completed fibers}
For any point $y\in Y$, the seminorm $\|\ \|_\omega$ induces seminorms on the stalk $\omega_{f,y}$ and the fiber $\omega_f(y)=\omega_{f,y}\otimes\kappa_G(y)$. We claim that they are compatible with the seminorm on the dualizing module of $\calH(y)/\calH(x)$.

\begin{lem}\label{fiberlem}
Let $f\:Y\to X$ be a generically \'etale morphism of nice $k$-analytic curves and $y\in Y$ a point. Then the maps $\omega_{f,y}\to\omega_f(y)\to\omega_{\calH(y)/\calH(x)}$ are isometries. In particular, the completed fiber $\wh{\omega_f(y)}$ coincides with $\omega_{\calH(y)/\calH(x)}$.
\end{lem}
\begin{proof}
Analogous results for K\"ahler seminorms on $\Omega_X$ and $\Omega_Y$ were proved in \cite[Theorem~6.1.8 and Corollary~6.1.9]{Temkintopforms}. Applying the $\calHom$ functor we obtain the assertion.
\end{proof}

Notice that the claim holds for points of type 1 as well. In this case, the seminorms on $\omega_{f,y}$ and $\omega_f(y)$ are zero, and hence $\wh{\omega_f(y)}$ equals $0=\omega_{k/k}$.

\subsubsection{A new interpretation of the different}
As an application of the above theory we obtain a very conceptual interpretation of $\delta_f$ in terms of the bivariant form $\tau_f$.

\begin{theor}\label{deltatau}
Let $f\:Y\to X$ be a generically \'etale morphism of nice $k$-analytic curves. Then $\delta_f=\|\tau_f\|$.
\end{theor}
\begin{proof}
Both functions are pm, hence it suffices to compare them at a point $y\in Y^\hyp$. By Lemma~\ref{fiberlem}, we can compute $\|\tau_f\|_y$ on the level of completed residue fields, that is, $\|\tau_f\|_y=\|\tau_{\calH(y)/\calH(x)}\|$. By Theorem~\ref{taufields}(i), the latter number equals $\delta_{\calH(y)/\calH(x)}$, which coincides with $\delta_f(y)$ by the definition of $\delta_f$.
\end{proof}

\subsubsection{Computation of the different}
Finally, we provide a convenient explicit formula for $\delta_f$, which is an analogue of Lemma~\ref{diflem0}(i). By $f^*(\|\phi\|_{X})=\|\phi\|_{X}\circ f$ we denote the pullback of the real-valued function $\|\phi\|_{X}$ to $Y$.

\begin{lem}\label{difflem}
Let $f\:Y\to X$ be a generically \'etale morphism of nice $k$-analytic curves, and let $\phi\in\Omega_X(X)$ be any non-zero differential form on $X$. Then $\delta_f=\|f^*(\phi)\|_Y/f^*(\|\phi\|_X)$.
\end{lem}
\begin{proof}
It suffices to check this at a point $y\in Y^\hyp$. Set $x=f(y)$ and $\psi=f^*(\phi)$, then $\delta_{\calH(y)/\calH(x)}=\|\psi(y)\|_{\hatOmega,\calH(y)}/\|\phi(x)\|_{\hatOmega,\calH(x)}$ by \cite[Theorem~2.4.4]{CTT}.
\end{proof}

\subsection{Reductions}

\subsubsection{General framework}
Recall that for any $x\in X$ of type 2, the curve $C_x$ naturally embeds into $X_G$. Given a seminormed $\calO_{X_G}$-module $\calF$ with unit ball $\calF^\di$, the reduction of $\calF_{C_x}:=\calF|_{C_x}$ is defined to be $\tilcalF_{C_x}:=\calF^\di_{C_x}\otimes_{\kcirc}\tilk$. Its stalk at the generic point $x$ will be denoted $\tilcalF_x$. As customary, the image of a section $s$ under reduction will be denoted $\tils\in\tilcalF_x$. These notions are especially useful for $|k^\times|$-pm seminorms because of the following lemma.

\begin{lem}\label{redlem}
Let $X$ be a nice curve and $(\calF,\|\ \|)$ an invertible $\calO_X$-module provided with a $|k^\times|$-pm seminorm. Then $\tilcalF_{C_x}$ is a lattice in $\tilcalF_x$ for any point $x\in X$ of type 2, and hence induces an order $\ord_v$ for any closed point $v\in C_x$. Furthermore, if $s$ is a section of $\calF$ on a $G$-neighborhood of $v$, then $\|s\|_x=|c^{-1}|$ for some $c\in k$ and the reduction $\wt{cs}\in\tilcalF_x$ can be used to compute the slopes of $\|s\|$ via $$\slope_v(\|s\|)=-\ord_v(\wt{cs}).$$
\end{lem}
\begin{proof}
We will prove both claims together. In the first one, no $s$ is given, so we simply fix a non-zero $s$. By definition, $\|s\|_x\in|k^\times|$ hence a required $c\in k$ exists. Multiplying $s$ by $c^{-1}$ does not affect $n=\slope_v(\|s\|)$, hence we can assume that $\|s\|=1$. Fix $t_v\in\calOcirc_{X,x}$ such that $\tilt_v$ is a uniformizer at $v$. Then $\slope_v(\|t_v\|)=-1$ and hence $s'=t_v^{n}s$ has slope 0 at $v$. In particular, if $I=[x,x']$ is a sufficiently small interval in the direction of $v$ then $\|s'\|=1$ along $I$. It follows that $s'$ generates $\calF^\di_v$ and hence $\tils'$ generates $\tilcalF_{C_x,v}$. In particular, $\tilcalF_{C_x,v}$ is invertible and hence defines an order $\ord_v$. Moreover, since $\tils'=\tilt_v^{n}\tils$, we obtain that $\ord_v(\tils)=-n$.
\end{proof}

\subsubsection{Reduction of the dualizing sheaf}
The above lemma in fact reduces to a tautological unraveling of definitions. An informative input is obtained by computing reductions of sheaves one studies. In our case, this is done as follows:

\begin{theor}\label{redomega}
Assume that $Y$ is a nice $k$-analytic curve and $y\in Y$ is a point of type 2. Then $\tilOmega_{Y_G,C_y}=\Omega^\rmlog_{C_y/\tilk}$. Furthermore, if $f\:Y\to X$ is a generically \'etale finite morphism of nice curves then $\tilomega_{f,C_y}=\omega^\rmlog_{\tilf_y}$.
\end{theor}
\begin{proof}
The first claim was proved in \cite[Lemma~4.5.2]{CTT}. By Corollary \ref{omegaball} we obtain that
$$\tilomega_{f,C_y}=\calHom_{\calOcirc_{Y_G,C_y}}(f^*(\Omega^\di_{X_G,C_y}),\Omega^\di_{Y_G,C_y})\otimes_{\kcirc}\tilk=\calHom_{\calO_{C_y}}(\tilf_y^*(\Omega^\rmlog_{C_x}),
\Omega^\rmlog_{C_y})=\omega^\rmlog_{\tilf_y}.$$
\end{proof}

As a corollary we can describe behaviour of slopes of the K\"ahler seminorm of differential forms at points of type 2.

\begin{cor}
Let $X$ be a nice $k$-analytic curve, $x\in X$ a point of type 2, and $\phi$ a differential form defined in a neighborhood of $x$. Then $m_v=\slope_v(\|\phi\|)+1$ equals zero for almost any $v\in C_x$, and the divisor $\calO(-\sum_{v\in C_x}m_vv)$ is rationally equivalent to $\Omega_{C_x/\tilk}$. In particular, if $x$ is an inner point then $\sum_{v\in C_x}m_v=2-2g(x)$.
\end{cor}
\begin{proof}
Choose $c\in k$ with $|c|=\|\phi\|_x$ and set $\psi=\wt{c^{-1}\phi}\in\Omega_{\wHx/\tilk}$. By Lemma~\ref{redlem}, $1-m_v$ is the order of $\psi$ at $v$ with respect to the reduction of $\Omega_{X_G,C_x}$ at $x$, which is $\Omega^\rmlog_{C_x/\tilk}$ by Theorem~\ref{redomega}. Thus, $m_v=1-\logord_v(\psi)=-\ord_v(\psi)$ is the usual differential order of $\psi$. The assertion follows.
\end{proof}

Absolutely in the same way one describes seminorms of bivariant forms, so we formulate the result and skip the proof. By $n_y$ we denote the multiplicity of $f$ at $y$, see \S\ref{convsec}.

\begin{cor}\label{slopebivar}
Let $f\:Y\to X$ be a generically \'etale morphism of nice $k$-analytic curves, $y\in Y$ a point of type 2, and $\phi\in\omega_f(U)$ a bivariant form defined in a neighborhood of $y$. Then $m_v=\slope_v(\|\phi\|)+1-n_v$ equals zero for almost any $v\in C_y$, and the divisor $\calO(-\sum_{v\in C_y}m_vv)$ is rationally equivalent to $\omega_{\tilf_y}$. In particular, if $y$ is an inner point and $x=f(y)$, then $\sum_{v\in C_y}m_v=2-2g(y)-n_y(2-2g(x))$.
\end{cor}

\begin{rem}
Theorem \ref{deltatau} and Corollary~\ref{slopebivar} immediately imply the local RH formula in Theorem~\ref{maindif}(ii). Moreover, if $g(y)>0$ then $\Pic^0(C_y)$ is non-trivial and this imposes a further restriction on the slopes.  In fact, we see that the slopes of $\delta_f=\|\tau\|_f$ are controlled by a finer invariant, the reduction of $\tau_f$. On the other hand, so far we have not obtained a new explanation of some other properties of the different, including the restrictions on its slopes from Theorem~\ref{slopeth}(iii). We will see in Remark~\ref{sloperem2} that they are related to the fact that the reduction of $\tau_f$ is a meromorphic bivariant form of a rather special type.
\end{rem}

\subsection{Reduction of the different}\label{reddifsec}
Our next goal is to study reduction of $\tau_f$.

\subsubsection{Scaled reduction}
For a point $y\in Y$ of type 2, take $c\in k$ with $|c|=\delta_f(y)$ and consider the meromorphic bivariant form $\tiltau_{f,y}=\wt{c^{-1}\tau_f}$ on $C_y$. We call $\tiltau_{f,y}$ the {\em scaled reduction} of $\tau_f$ at $y$; it is defined up to multiplying by a constant $a\in\tilk^\times$. The notation using tilde is slightly abusing but this will not lead to any confusion. 
Let us summarize what we know about $\tiltau_{f,y}$ so far.

\begin{lem}\label{tiltaulem}
With the above notation

(i) $\slope_v(\delta_f)=-\logord_v(\tiltau_{f,y})=-\ord_v(\tiltau_{f,y})+n_v-1$,

(ii) $\tiltau_{f,y}=\tiltau_{\calH(y)/\calH(x)}$,

(iii) $\tiltau_{f,y}=(\tilphi)'\otimes\tilpsi\in(\tilf_y^*\Omega_{\wHx/\tilk})'\otimes\Omega_{\wHy/\tilk}$, where $\phi\in\hatOmega_{\calH(x)}$ is any form satisfying $\|\phi\|=1$, and $\psi\in\hatOmega_{\calH(y)}$ is the pullback of $c\phi$, where $c\in k$ is such that $|c^{-1}|=\delta_f(y)$.
\end{lem}
\begin{proof}
Indeed, (i) follows from Theorem~\ref{redomega} and Lemma~\ref{redlem}, (ii) follows from Lemma~\ref{fiberlem}, and (iii) follows from (ii) and Theorem~\ref{taufields}(ii).
\end{proof}

\subsubsection{The tame case}
Now, let us compute $\tiltau_{f,y}$ more specifically. In the tame case, the trace form is compatible with the reduction:

\begin{lem}\label{tautamelem}
Keep the above notation. If $f$ is residually tame at $y$ then $\delta_f(y)=1$ and $\tiltau_{f,y}=\tau_{\tilf_y}$.
\end{lem}
\begin{proof}
Since $\delta_f(y)=1$ by Theorem~\ref{slopeth}(ii), the reduction does not involve any rescaling. By Lemma~\ref{tiltaulem}(iii), locally at $y$ we can present $\tau_f$ as $\phi^{-1}\otimes f^*(\phi)$, where $\|\phi\|_x=1$. Hence $\tiltau_{f,y}=\tilphi^{-1}\otimes \tilf_y^*(\tilphi)$ is the trace form $\tau_{\tilf_y}$ of the generically \'etale morphism $\tilf_y$.
\end{proof}

Note that the lemma conceptually explains why the assertion of Theorem~\ref{maindif}(2) reduces to the RH formula for $\tilf_y$ in the residually tame case.

\subsubsection{The degree-$p$ case}
Next we study the simplest residually wild case, the case of degree $p$. 

\begin{theor}\label{degpth}
Let $L/K$ be a separable wildly ramified extension of degree $p$ of one-dimensional analytic $k$-fields of type 2. Choose a tame monomial parameter $z\in K$ such that $|z|=1$ and set $\tilw=\tilz^{1/p}\in\wHy$. Let $z_L$ denote $z$ viewed as an element of $L$, and let $\tilt$ be the tame reduction of $z_L$. Moreover, if $z_L$ is a mixed parameter then choose $\tilt$ to be the canonical tame reduction. Then $\tiltau_{L/K}$ is equal to:
\begin{itemize}
\item[(i)] $(d\tilz)'\otimes d\tilt$ if $z$ is a pure parameter in $\calH(y)$,

\item[(ii)] $(d\tilz)'\otimes(d\tilt+\tilw^{p-1}d\tilw)$ otherwise.
\end{itemize}
In addition, case (ii) takes place if and only if $\delta_{L/K}=|p|$, in particular, $k$ is of mixed characteristic.
\end{theor}
\begin{proof}
Fix $c\in k^\times$ such that $|c|=\delta_{L/K}^{-1}$. Furthermore, take $c=p^{-1}$ if $\delta_f(y)=|p|$. We will prove the theorem by applying Lemma~\ref{tiltaulem}(iii) to the form $\phi=dz$. Since $z$ is a tame monomial parameter at $x$ we have that $\|\phi\|_x=1$ and $0\neq\tilphi=d\tilz$. Let $\psi\in\hatOmega_L$ be the image of $c\phi$, then in view of Lemma~\ref{tiltaulem}(iii) we only need to prove that $\tilpsi$ equals $d\tilt$ or $d\tilt+\tilw^{p-1}d\tilw$, respectively.

Fix a presentation $z_L=w^p+t_1$ in $L$, where $t_1$ is the tame term if $z_L$ is a parameter, and $|t_1|<|p|$ otherwise. This fits the notation of the theorem since $\tilz_L^{1/p}$ is the reduction of $w$. Note that $\psi=cdz_L=cpw^{p-1}dw+cdt_1$ and $\|dt_1\|_y=|t_1|_y$. Since $\tilz\notin\wHx^p$ we have that $\tilw\notin\wHy^p$ and hence $w$ is a tame monomial parameter at $y$. In particular, $\|dw\|_y=|w|_y=1$ and hence $\|pw^{p-1}dw\|_y=|p|$. If $z$ is a pure parameter then $|t_1|_y>|p|$ and hence $|c^{-1}|=\delta_f(y)=|t_1|_y>|p|$. Setting $t=ct_1$ we obtain that $\wt{\psi}=\wt{cdt_1}=\wt{dt}=d\tilt$, and clearly $\tilt$ is a tame reduction of $z_L$.

If $z$ is not a pure parameter then $|t_1|_y\le|p|$ and hence $\delta_f(y)=\|c^{-1}\psi\|_y\le|p|$. By Theorem~\ref{slopeth}(ii) we must have $\delta_f(y)=|p|$. Thus $c=p^{-1}$ and ${\psi}={w^{p-1}dw+dt}$ where $t=p^{-1}t_1$. In particular, $\wt{\psi}=d\tilt+\tilw^{p-1}d\tilw$ and $\tilt$ is the canonical tame reduction of $z_L$ in this case.
\end{proof}

\subsubsection{Mixed bivariant forms}
Theorem~\ref{degpth} motivates the following definition. Let $F/E$ be a finite extension of fields of characteristic $p$ and assume that $[E:E^p]=p$. In particular, $\Omega_E$ and $\Omega_F$ are one-dimensional. A bivariant form $(dx)'\otimes dt\in\omega_{F/E}=(\Omega_E)'\otimes_E\Omega_F$ will be called {\em exact}. If $E=F^p$ then by a {\em mixed} bivariant form we mean any $\phi\in\omega_{F/E}$ of the form $(d_Ew^p)'\otimes(d_Ft+w^{p-1}d_Fw)$, where $t\in F$ and $w\in F\setminus E$. We have proved in Theorem~\ref{degpth} that the reduction of a trace form is always exact or mixed. The converse also holds:

\begin{theor}\label{type2fields}
Let $K$ be a one-dimensional $k$-field of type 2 and let $F=\tilK^{1/p}$ be the purely inseparable extension of $\tilK$ of degree $p$. Let $\phi\in\omega_{F/\tilK}$ be a non-zero bivariant form for $F/\tilK$, and let $\delta_0\in|k^\times|$. Assume that either $|p|<\delta_0<1$ and $\phi$ is exact, or $\delta_0=|p|$ and $\phi$ is mixed. Then there exists a wildly ramified extension $L/K$ of degree $p$ such that $\delta_{L/K}=\delta_0$ and $\phi=\tiltau_{L/K}$ under the identification $\tilL=F$.
\end{theor}
\begin{proof}
Fix any field $L$ of type 2 with $\tilL=F$. For example, one can take $L$ to be any wildly ramified extension of $K$ of degree $p$. (In fact, it is easy to see that the $k$-isomorphism class of $L$ is determined by $\tilL$, so the choice is not essential here.) We will construct an embedding $K\into L$ satisfying assertions of the lemma.

By definition, either $\phi=(d\tilz)'\otimes d\tilt$ or $\phi=(d\tilz)'\otimes (d\tilt+\tilw^{p-1}d\tilw)$, where $\tilz\in\tilK\setminus\tilK^p$, $\tilt\in\tilL\setminus\tilL^p$, and $\tilw=\tilz^{1/p}$. Choose liftings $z\in K$ and $t,w\in L$, and note that they are tame monomial parameters by Lemma~\ref{tameparamlem}. In particular, $K$ is tame, and hence unramified, over $K_0=\wh{k(z)}$.

Fix $c\in k^\times$ such that $|c|=\delta_0$ and $c=p$ if $\delta_0=|p|$, and consider the element $v=w^p+ct$. Since $\tilv=\tilw^p$ is transcendental over $\tilk$, there is an isomorphism $K_0=\wh{k(v)}$ sending $z$ to $v$, and we obtain an embedding $K_0\into L$ whose reduction is the embedding $\tilK_0=\tilk(\tilz)\into\tilL$. Since $K/K_0$ is unramified and $\tilK$ is a separable subextension of $\tilL/\tilK_0$, the embedding $K_0\into L$ factors through an embedding $i\:K\into L$. If we identify $K$ with the image of $i$, then $z=w^p+ct$, and by Theorem~\ref{degpth} we obtain that $\tiltau_{L/K}=\phi$, as required.
\end{proof}

\begin{rem}\label{sloperem2}
(i) Mixed forms are combined from an exact part and a part $\lam=(d_Ew^p)'\otimes w^{p-1}d_Fw$, which reveals a logarithmic behaviour. It can be identified with the section $(\frac{dw}{w})^{\otimes(1-p)}$ of $\Omega_F^{\otimes(1-p)}$. In particular, all its zeros are of order $p-1$, and all poles are of order divisible by $p-1$.

(ii) Assume that $n_f(y)=p$. Since any exact/mixed form can be obtained as the reduction of $\tau_{f,y}$, it is possible to deduce all properties of the different from the special form of $\tau_{f,y}$. For example, let us explain the properties from Remark~\ref{sloperem}. If $|p|<\delta_f(y)<1$ then $\tiltau_{f,y}$ is exact. It follows that $\ord_v\tiltau_{f,y}\notin -1+p\bfZ$ for any $v\in C_y$, hence $\slope_v\delta_f\notin p\bfZ$ by Lemma~\ref{tiltaulem}(i). If $\delta_f(y)=|p|$ then $\phi=\phi_0+\lam$ is mixed, and one can check that its order is bounded by $p-1$ (in the extreme case the zero comes from $\lam$ and $\phi_0$ vanishes to a higher order). In particular, the slopes of $\delta_f$ at $y$ are non-negative, and hence $|p|$ is the minimal possible value of $\delta_f$.
\end{rem}

\section{\'Etale annular $p$-covers}\label{annulisec}
In this section we study wild covers of annuli of degree $p$.

\subsection{Annular covers}

\subsubsection{Annuli, skeletons and coordinates}\label{anncoordsec}
Let $Y=\calM(\calA)$ be a closed annulus of exponential modulus $r\in(0,1)$. The boundary $\partial(Y)=\{a,b\}$ consists of two points and the interval $l_Y=[a,b]$ is the minimal skeleton of $Y$. We will call $l=l_Y$ {\em the skeleton} of $Y$. An ordering $(a,b)$ of the boundary will be called an {\em orientation} on $Y$. Orienting $Y$ is equivalent to orienting the skeleton. For any real-valued function $\phi\:Y\to\bfR$ we will denote by $\phi_l\:l\to\bfR$ its restriction on $l$. For example, any $f\in \calA$ induces a function $|f|_l\:l\to\bfR_{\ge 0}$.

By a {\em monic coordinate} on $Y$ we mean any element $y\in\calA$ inducing on isomorphism of $Y$ onto the standard annulus $A(r,1)$ of radii $r\le 1$ and 1. Equivalently, $\calA\simeq k\{y,ry^{-1}\}$. Any such coordinate induces an isomorphism of $|k^\times|$-pm spaces $l\simeq [r,1]$. For an oriented annulus $Y$ we will only consider monic coordinates inducing an oriented isomorphism $l\simeq [r,1]$.

For completeness, we will also allow the case of exponential modulus $1$. Then $\calA=k\{y,y^{-1}\}$, the boundary consists of a single point $q$, and orienting $Y$ is equivalent to ordering the set of two infinite points of $C_q=\Spec(\tilk[\tily^{\pm 1}])=\bfG_{m,\tilk}$.

\subsubsection{Units and domination}
Given elements $u,v\in\calA$ with $u$ a unit, we say that $u$ {\em strictly dominates} $v$ and write $v\prec u$ if $|v/u|_\calA<1$, where $|\ |_\calA$ denotes the spectral norm on $\calA$. This happens if and only if $|v/u|_q<1$ for any $q\in Y$ if and only if $|v/u|_q< 1$ for any $q\in\partial(Y)$. In particular, $v\prec u$ if and only if $|v|_l<|u|_l$ as functions on $l$.

Fix a monic coordinate $y$. It is easy to see that an element $u$ is a unit if and only if the series $u=\sum_{i\in\bfZ}a_iy^i$ contains a {\em dominant term} $a_ny^n$, i.e. a term that strictly dominates all other terms. (This corrects the inaccurate formulation in \cite[Lemma~3.5.8(i)]{CTT}, where the spectral norm was used instead of domination.) In the sequel, we will need the following computation:

\begin{lem}\label{domlem}
Assume that $u,z\in\calA$ are such that $u$ is a unit strictly dominating $z$, and $i$ is an integer. Then $(u+z)^i-u^i=u^{i-1}zQ$, where $|Q|_\calA\le 1$.
\end{lem}
\begin{proof}
Dividing by $u^i$ we can assume that $u=1$. It suffices to show that $zQ=(1+z)^i-1=\sum_{l=1}^\infty a_lz^l$ with $a_l\in\bfZ$. If $i\ge 0$ this is obvious, and for $i<0$ the claim follows by expanding the right hand side of $(1+z)^i=(1-z+z^2-\dots)^{-i}$.
\end{proof}

\subsubsection{Annular covers}\label{anncovsec}
By an {\em annular $m$-cover} we mean a finite morphism $f\:Y\to X$ of degree $m$ between annuli. Using monic coordinates $y$ and $x$ on $Y$ and $X$ it is described by a series $x=\varphi(y)=\sum_{i\in\bfZ}c_iy^i$, that will be called a {\em presentation} of $f$. Since $x$ is a unit, there is a dominant term $c_dy^d$. In particular, $|f|_l=|c_dy^d|_l$ is a monomial function of slope $d$ on $l$, and $r(X)=r(Y)^m$. Moreover, $|c_d|=1$ and replacing $x$ by $c_d^{-1}x$ we can assume that $c_d=1$.

Furthermore, the absolute value of $d$ equals $m$, and the sign of $d$ indicates whether $f$ is compatible with the orientations induced by the coordinates. In particular, choosing compatible orientations we can and always will assume that $d=m$. A presentation with dominant term $y^m$ will be called a {\em monic presentation}.

\subsubsection{Kummer and binomial covers}\label{Kumsec}
Let $X=\calM(k\{x,rx^{-1}\})$ be a closed annulus with a fixed monic coordinate $x$, and let $g\:Y\to X$ be a residually tame annular covering. Then $m:=\deg(g)\notin p\bfZ$ and, using that the radius of convergence of $(1+t)^{1/m}$ is 1, one easily obtains that $y=x^{1/m}$ is a monic coordinate on $Y$. In particular, $Y$ is $X$-isomorphic to the {\em Kummer covering} of degree $m$, i.e. the annular covering $\calM(k\{y,r^{\frac{1}{m}}y^{-1}\})\to X$.\footnote{In fact, any tame \'etale covering of $X$ is Kummer by \cite[Theorem~6.3.5]{berihes}. The difficult part is to prove that the cover is annular.}

Kummer covers are given by monomials. By a {\em standard binomial} $m$-cover we mean a cover $A(r,1)\to A(r^m,1)$ given by $y\mapsto y^m+c_ny^n$. A {\em binomial cover} is an annular cover $f\:Y\to X$ isomorphic to a standard one. In other words, $f$ admits a monic presentation of the form $\varphi(y)=y^m+c_ny^n$.

\begin{rem}
(i) An isomorphism of morphisms is given by compatible isomorphisms between targets and sources. In our situation, this amounts to choosing coordinates both on $Y$ and $X$. Our main result about \'etale annular $p$-covers will be that it is either Kummer or binomial.

(ii) Unlike the tame case, it is crucial to play with both coordinates. In particular, the above result is wrong once a coordinate on $X$ is fixed. The set of $X$-isomorphism classes of \'etale annular $p$-covers of $X$ is huge.
\end{rem}

\subsubsection{\'Etale annular covers and the different}\label{etalecovsec}
Let $f\:Y\to X$ be an annular $m$-cover with a monic presentation $x=\varphi(y)=\sum_{i\in\bfZ}c_iy^i$. Then $f$ is \'etale if and only if the derivative $\varphi'(y)=\sum_{i\in\bfZ}ic_iy^{i-1}$ is a unit, that is, there is a dominant term $nc_ny^{n-1}$. In the tame case, $m\notin p\bfZ$ and one automatically has that $n=m$.

By \cite[Theorem~4.1.6]{CTT}, the restriction $\delta_l$ of $\delta$ onto the skeleton $l=l_Y$ coincides with the restriction of $|yx^{-1}\varphi'|$ onto $l$. Since $|x|=|y^m|$ on $l$ we obtain that $\delta_l$ coincides with the norm of the unit $y^{1-m}\varphi'(y)=\sum_{i\in\bfZ} ic_iy^{i-m}$ and hence coincides with the norm of the dominant term $nc_ny^{n-m}$. In particular, $n-m$ is the slope of the different on $l$.

\subsection{Metrization of $\Aut(A(r,1))$}

\subsubsection{The group of automorphisms}
Consider a standard annulus $A(r,1)=\calM(\calA)$, where $\calA=k\{y,ry^{-1}\}$, and let $G=G(r)$ be its group of automorphisms. Once the coordinate is fixed, we can identify elements $\phi\in G$ with their presentations $\varphi(y)$, and we will not distinguish them. This identifies $G$ with the set of series $\phi=\sum_{i\in\bfZ}a_iy^i$ possessing a dominant term $a_ny^n$ with $n\in\{\pm 1\}$ and $|a_ny^n|_\calA=1$ (i.e. either $n=1, |a_1|=1$ or $n=-1,|a_{-1}|=r$), and the operation corresponds to the composition.

\subsubsection{Composition}
In fact, one can compose automorphisms $\phi$ with arbitrary elements $g=\sum_{i\in\bfZ}g_iy^i$. Namely, $g(\phi(y))=\sum_ig_i\phi(y)^i$ is a well defined element of $\calA$. Even more generally, the composition $g\circ h$ is defined when $h\in\calA$ has a dominant term $h_ny^n$ with $n\neq 0$ and $|h_ny^n|_\calA=1$, and hence can be viewed as a morphism $A(r,1)\to A(r^{|n|},1)=\calM(\calB)$ for $\calB=k\{x,r^{|n|}x^{-1}\}$, and $g(x)\in\calB$ is a function on $A(r^{|n|},1)$.

\subsubsection{The metric}
The group $G_+$ of orientation preserving automorphisms is given by $n=1$. On this group we introduce a metric by setting $\|\phi\|=|y^{-1}\phi-1|_\calA$. Note that $\phi=a_1y+\lam y$, where $a_1\in k$, $|a_1|=1$ and $|\lam|_\calA<1$, and the value of $\|\phi\|\in[0,1]$ measures how far $\phi$ is from the identity automorphism $y$. The subset of $G_+$ defined by $\|\phi\|<1$ will be denoted $\Gcirccirc$, it is characterized by the inequality $|a_1-1|<1$. Finally, we naturally extend the metric to $G$ by setting $\|\phi\|=1$ for any $\phi\in G\setminus G_+$.

\begin{lem}\label{autlem}
(i) The set $\Gcirccirc$ is a normal subgroup of $G$, and mapping $\phi=\sum_i c_iy^i$ to $\tilc_1$ induces an isomorphism $G_+/\Gcirccirc\toisom\tilk^\times$.

(ii) Let $\phi\in \Gcirccirc$, $h\in\calA$, and $u\in\calA^\times$. Then $|u^{-1}(h\circ\phi-h)|_\calA\le|u^{-1}h|_\calA\cdot\|\phi\|$.

(iii) Sending $\phi$ to $\alp(\phi):=y^{-1}\phi-1$ establishes an isometric bijection $\Gcirccirc\to\calA^\circcirc$, which is an approximation to a homomorphism in the following sense: for any $\phi,\psi\in\Gcirccirc$ one has that $$|\alp(\psi\circ\phi)-\alp(\phi)-\alp(\psi)|_\calA\le|\alp(\phi)|_\calA\cdot|\alp(\psi)|_\calA.$$
\end{lem}
\begin{proof}
(i) Consider the restriction homomorphism $\rho\:G\to\Aut(\tilcalA)$. Note that $\tilcalA=\tilk[\tily,\tilz]/(\tily\tilz)$, where $z=ay^{-1}$ with $a\in k$, $|a|=r$. A direct computation of $\tilg:=\rho(g)$ yields: $\tilg(\tily)=\tilc_1\tily$ and $\tilg(\tilz)=\tilc_1^{-1}\tilz$ for any $g\in G_+$, and $\tilg(\tily)=\tilb\tilz$ and $\tilg(\tilz)=\tilb\tily$ , where $b=c_{-1}/a$, for any $g\in G\setminus G_+$. In particular, $\Gcirccirc=\Ker(\rho)$ and the claim follows.

(ii) If $ay^n$ is the dominant term of $u$ then $u=ay^nw$, where $w$ is a unit such that $|w|_\calA=|w^{-1}|_\calA=1$. Therefore we can replace $u$ by $y^n$ in the assertion. Let $h=\sum_i c_iy^i$. Since $|y^{-n}h|_\calA=\max_i|c_iy^{i-n}|_\calA$, it suffices to prove that for any $i$ $$|y^{-n}c_i(\phi^i-y^i)|_\calA\le|c_iy^{i-n}|_\calA\cdot\|\phi\|.$$ We have that $\phi=y(1+v(y))$, where $v=\alp(\phi)$ and $|v|_\calA=\|\phi\|$. Hence $$|y^{-n}c_i(\phi^i-y^i)|_\calA=|c_iy^{i-n}((1+v)^i-1)|_\calA\le|c_iy^{i-n}|_\calA\cdot|(1+v)^i-1|_\calA,$$ and it remains to note that $|(1+v)^i-1|_\calA\le|v|_\calA$ by Lemma~\ref{domlem}.

(iii) Clearly, $\Gcirccirc\to\calA^\circcirc$ is a bijection, and it is an isometry by the definition of $\|\ \|$. Set $h=\psi-y$ and note that $h\prec y$. We have that $\psi\circ\phi-\psi-\phi=h\circ\phi-h-y$ and hence $\alp(\psi\circ\phi)-\alp(\phi)-\alp(\psi)=y^{-1}(h\circ\phi-h)$. It remains to note that $|y^{-1}h|=\|\psi\|$ and use (ii) with $u=y$.
\end{proof}

The lemma implies the following result.

\begin{cor}\label{invnormcor}
The function $\|\ \|$ makes $G$ a complete non-archimedean group with an open subgroup $\Gcirccirc$:

(i) $\|\ \|$ is symmetric: $\|\phi\|=\|\phi^{-1}\|$,

(ii) $\|\ \|$ is non-archimedean: $\|\phi\circ\psi\|\le\max(\|\phi\|,\|\psi\|)$.
\end{cor}
\begin{proof}
(i) If $\phi\notin\Gcirccirc$ then $\|\phi\|=\|\phi^{-1}\|=1$, so assume that $\phi\in\Gcirccirc$. It suffices to show that $\|\phi^{-1}\|\le \|\phi\|$. Lemma~\ref{autlem}(iii) applied to $\phi$ and $\psi=\phi^{-1}$ yields $$|\alp(\phi)+\alp(\phi^{-1})|_\calA\le|\alp(\phi)|_\calA\cdot|\alp(\phi^{-1})|_\calA\le|\alp(\phi)|_\calA.$$ This implies that $\|\phi^{-1}\|=|\alp(\phi^{-1})|_\calA\le|\alp(\phi)|_\calA=\|\phi\|$, as required.

(ii) This is deduced from Lemma \ref{autlem}(iii) in a similar fashion.

Completeness of $\Gcirccirc$ is clear from Lemma \ref{autlem}. By definition, $\|\phi\|\le 1$ for any $\phi\in G$ and the inequality is strict if and only if $\phi\in\Gcirccirc$. So, $\Gcirccirc$ is an open and closed subgroup, and therefore $G$ is complete too.
\end{proof}

The following remark will not be used, so the reader can skip it.

\begin{rem}
(i) It is easy to see that the inequality of Lemma~\ref{autlem}(ii) extends to the whole $G_+$.

(ii) Informally speaking, $\Gcirccirc$ behaves as the maximal subgroup of $G$ with a pro-unipotent model over $\kcirc$.
\end{rem}

\subsection{Classification of \'etale annular $p$-covers}

\begin{notation}\label{fnot}
We now restrict to the case of $p$-covers. So, until the end of Section \ref{annulisec}, $f\:Y\to X$ denotes an annular $p$-cover. By $l=[r,1]$ we denote the skeleton of $Y$. We orient $Y$ so that the monomial function $\delta|_l$ is non-decreasing, and orient $X$ compatibly. By $x=\varphi(y)=\sum_{i\in\bfZ}c_iy^i$ we will always denote a monic presentation of $f$, and then $nc_ny^{n-1}$ will denote the dominant term of $\varphi'(y)$.
\end{notation}

\begin{lem}\label{tametermlem}
Keep the above notation. Then $n\ge p$ and one of the following possibilities holds:

(1) $n=p$ and $k$ is of mixed characteristic.

(2) $(n,p)=1$ and $1>|c_ny^{n-p}|_l>|p|$.
\end{lem}
\begin{proof}
By \S\ref{etalecovsec} the non-decreasing function $\delta_l$ equals $|nc_ny^{n-p}|_l$. Hence $n\ge p$. In the mixed characteristic case, the free term $p$ of $y^{1-p}\varphi'(y)$ dominates any term $ic_iy^{i-p}$ with $i\in p\bfZ$. Therefore, either (1) holds or $(n,p)=1$. In the second case, $y^p\succ c_ny^n$ and $c_n y^{n-1}\succ py^{p-1}$, hence $1>|c_ny^{n-p}|_l>|p|$.
\end{proof}

\subsubsection{The different}
For brevity, a monomial function $h=|ct^s|$ on an interval $I$ will be called {\em relevant} if one of the following two possibilities holds:

(1) $h(I)=|p|>0$,

(2) $h$ is increasing, $h(I)\subset(|p|,1)$, and $(s,p)=1$. 

\begin{lem}\label{relmonlem}
With Notation \ref{fnot}, $\delta_l$ is a relevant monomial on $l$.
\end{lem}
\begin{proof}
We observed in \S\ref{etalecovsec} that $\delta_l=|nc_ny^{n-p}|_l$. Hence the assertion follows from Lemma~\ref{tametermlem}.
\end{proof}

\subsubsection{Dominant tame term}\label{domtamesec}
We will later see that cases (1) and (2) in Lemma~\ref{tametermlem} correspond to Kummer and binomial covers, respectively. In case (2), we call $c_ny^n$ the {\em dominant tame term} of $\varphi$. It strictly dominates any other term $c_iy^i$ with $(i,p)=1$. In case (1), the dominant tame term is zero by definition. In this case, $py^{p}$ strictly dominates any term $c_iy^i$ with $(i,p)=1$.

In view of the following lemma, by a {\em dominant tame term} of $f$ we mean any monomial of the form $cc_ny^n$ with $c\in k$ and $|c|=1$.

\begin{lem}\label{tametermlem2}
Let $f\:Y\to X$ and $\varphi(y)$ be as in Notation \ref{fnot}, and let $t$ be the dominant tame term of $\varphi$. Then,

(i) Any other monic presentation of $f$ has a dominant tame term of the form $ct$, where $c\in k$ and $|c|=1$.

(ii) Conversely, for any $c\in k$ with $|c|=1$ there exists a presentation whose dominant tame term is $ct$.
\end{lem}
\begin{proof}
The case $\delta_l=|p|$ is obvious. In the sequel we assume that $\delta_l>|p|$, hence the dominant tame terms are non-zero. Then the presentation in (ii) is obtained by the coordinate change $x'=c^{p/(p-n)}x$ and $y'=c^{1/(p-n)}y$.

Let us prove (i). Note that $\delta_l$ is an invariant of $f$. Writing $t=c_ny^n$ we see that $\delta_l=|c_ny^{n-p}|$ hence $|c_n|$ is determined by $f$, and $n$ is determined by $f$ when $r<1$. If $r=1$ then a more refined argument is needed. For example, $\partial(Y)=\{q\}$ is a single point and $C_q=\Spec(\tilk[\tily^{\pm 1}])$. A direct computation shows that $\tiltau_{f,y}=(d\tily^p)'\otimes d\tily^n$, and this easily implies that $n$ is an invariant of $f$.
\end{proof}

\subsubsection{The main theorem}
The following result completely classifies isomorphism classes of \'etale annular $p$-covers. Its main part is that any such cover is either Kummer or binomial.

\begin{theor}\label{binomth}
Assume that $f\:Y\to X$ is an \'etale annular $p$-cover. Then $f$ admits a presentation $x=y^p+t$, where $t$ is a dominant tame term of $f$. Conversely, any Kummer or binomial presentation of $f$ is of this form.
\end{theor}

The proof of this theorem requires some computations and will occupy Section~\ref{proofsec}. In view of \ref{domtamesec} we immediately obtain the following

\begin{cor}\label{classcor}
Fix $r\in|k^\times|$ with $r<1$ and consider the set $C_r$ of isomorphism classes of \'etale annular $p$-covers $f\:Y\to X$ such that $r$ is the exponential modulus of $Y$. Then the correspondence $f\mapsto\delta_f|_l$ induces a bijection of $C_r$ onto the set of relevant monomials on $l=[r,1]$.
\end{cor}

\subsection{Proof of Theorem \ref{binomth}}\label{proofsec}

\subsubsection{The general line}
In view or Lemma~\ref{tametermlem2}, it suffices to show that $f$ is either Kummer or binomial. Let $X=\calM(\calB)$ and $Y=\calM(\calA)$. Fix an initial monic presentation $x=\varphi_0(y)$ of $f$. Our aim is to change both coordinates making the presentation binomial, but it will be convenient to fix the isomorphisms
$\calB=k\{x,r^px^{-1}\}$ and $\calA=k\{y,ry^{-1}\}$ and identify the automorphisms of these algebras with power series. In this language we should find automorphisms $h(x)\in \calB$ and $g(y)\in\calA$ of $\calB$ and $\calA$, respectively, so that the new presentation $\varphi=h\circ\varphi_0\circ g$ is of the form $y^p+cy^n$. We will construct $h$ and $g$ via a converging series of iterations, $g=g_0\circ g_1\circ\dots$ and $h=\dots\circ h_1\circ h_0$ such that $g_i$ and $h_i$ tend to the identities. The product will then converge by Corollary~\ref{invnormcor}.

\subsubsection{Two-term decompositions}
Let $\varphi=\sum_{i\in\bfZ}c_iy^i$ be a monic presentation of $f$ with dominant tame term $t$. By a {\em two-term decomposition} of $\varphi$ we mean a decomposition $\varphi(y)=\psi(y^p)+\lam(y)$, where $\psi(y^p)=\sum_{i\in \bfZ}a_iy^{pi}$ and the following condition holds: (1) if $t=0$ then $\lam\prec py^p$, (2) if $t=c_ny^n\neq 0$ then $t$ is the dominant term of $\lam$. By \S\ref{domtamesec}, any presentation possesses the natural two-term decomposition obtained by separating the terms with $i$ divisible and non-divisible by $p$. However, it will be convenient to use other decompositions too. We say that the decomposition is {\em simple} if $\psi(y^p)=y^p$.

\begin{lem}\label{palplem}
For any element $\alp(y)\in\calA$ there exists a decomposition $\alp(y)^p=\alp_1(y^p)+p\alp_2(y)$ such that $|\alp_1(y^p)|_\calA\le|\alp(y)|_\calA^p$ and $|\alp_2(y)|_\calA\le|\alp(y)|_\calA^p$.
\end{lem}
\begin{proof}
If $\alp=\sum_{i\in\bfZ} c_iy^i$ then one can take $\alp_1=\sum_{i\in\bfZ}c_i^py^{i}$.
\end{proof}

\subsubsection{Estimating the error}
Viewing $\psi$ and $\lam$ as elements of $\calB$ and $\calA$, respectively, set $$s_p=s_p(\psi):=|x^{-1}\psi(x)-1|_\calB=|y^{-p}\psi(y^p)-1|_\calA.$$ In addition, set $s=s(\lam):=|p^{-1}y^{-p}\lam|_\calA$ in case (1), and $s=s(\lam):=|t^{-1}\lam-1|_\calA$ in case (2). Note that  $s<1$ and $s_p<1$ for any two-term decomposition of a monic presentation, and if $s=s_p=0$ then $\varphi=\psi+\lam=y^p+t$ is either Kummer or binomial. Our strategy will be to alternate coordinate changes of $y$ that ``improve'' $\lam$ and reduce $s$, and coordinate changes of $x$ that ``improve'' $\psi$ and reduce $s_p$.

\subsubsection{Improving $\psi(y^p)$}\label{hsec}
We simply define $h(x)$ to be the inverse of $\psi(x)$ in the group $G(r^p)^{\circ\circ}=\Aut(\calB)^\circcirc$. Also, set $h_0(x)=h(x)-x$. Until the end of Section~\ref{annulisec}, notation like $\varphi'$ denotes another presentation of $f$ rather than the derivative of $\varphi$.

\begin{lem}\label{psilem}
Assume that $\varphi=\psi+\lam$ is a presentation of $f$ with a two-term decomposition, and let $h$ be as above. Then $\varphi'(y)=h(\varphi(y))$ has a simple two-term decomposition $y^p+\lam'$ such that $s'=s(\lam')\le\max(s,s_p)$.
\end{lem}
\begin{proof}
Setting $\rho=h_0(\varphi)-h_0(\psi)$ and $\lam'=\lam+\rho$ we have that $$\varphi'=\varphi+h_0(\varphi)=\psi+\lam+h_0(\psi)+\rho=h(\psi)+\lam+\rho=y^p+\lam'.$$ We claim that this is a required two-term decomposition, and to show this we should somehow control $\rho$. In fact, it is easy to see that the claim (and the lemma) will follow once we show that $\rho=b\lam$ with $|b|_\calA\le s_p$.

If $h_0(x)=\sum_{i\in\bfZ}a_ix^i$ then $\rho=\sum_{i\in\bfZ}\rho_i$, where $\rho_i=a_i((\psi+\lam)^i-\psi^i)$. As $\psi$ strictly dominates $\lambda$ by the definition of the two-term decomposition, we have $\rho_i=a_i\psi^{i-1}\lambda Q_i$ with $|Q_i|_\calA\le 1$ by Lemma \ref{domlem}. So, $b$ is the sum of $b_i=a_i\psi^{i-1} Q_i$ and it suffices to prove that $|b_i|_\calA\le s_p$. Since $\|h\|=\|\psi\|$ by Corollary~\ref{invnormcor}(i), $$|x^{-1}h_0|_\calB=||h||=||\psi||=|x^{-1}\psi(x)-1|_\calB=s_p.$$ Therefore  $|a_i|\le s_p$  for $i>0$ and $|a_i|\le r^{pi-p}s_p$ for $i\le 0$. Since $|\psi(y^p)^i|_\calA=|y^{ip}|_\calA$ for any $i\in\bfZ$, we obtain that $|a_i\psi^{i-1}|_\calA \le s_p$ and hence $|b_i|_\calA\le s_p$, as required.
\end{proof}

For the sake of simplicity, in the next two lemmas we restrict to the case of simple decompositions.

\subsubsection{Improving $\lam$ in case (1)}\label{case1sec}
Set $\alpha=-p^{-1}y^{-p}\lam$ and $g(y)=y+\alpha y$. Clearly $g$ is an element of $G(r)^{\circ\circ}=\Aut(\calA)^\circcirc$ satisfying $\|g\|=|\alpha|_\calA=s$.

\begin{lem}\label{lamlem1}
Assume that $\varphi=y^p+\lam$ is a simple two-term decomposition in case (1), and let $g(y)$ be as above. Then $\varphi'=\varphi(g(y))$ has a two-term decomposition $\psi'+\lam'$ such that $s'_p=s_p(\psi')\le s^p$ and $s'=s(\lam')\le s^2$.
\end{lem}
\begin{proof}
We have that $\varphi'(y)=g^p+\lam(g)=y^p(1+\alpha)^p+\lam(g)$. Taking $\alpha^p=\alp_1+p\alp_2$ as in Lemma~\ref{palplem} we obtain $$(1+\alpha)^p=1+p\alpha+p\alpha^2 Q +\alp_1+p\alp_2$$ for some $Q\in \calA^\circ$. By the definition of $\alpha$ we have $py^p\alpha=-\lam(y)$, hence $\varphi'=\psi'+\lam'$ with $\psi'=y^p(1+\alp_1)$ and $$\lam'=py^p(\alpha^2 Q+\alp_2)+\lam(g(y))-\lam(y).$$

By the construction, all terms of $\psi'$ are $p$-th powers of $y$ and $|y^{-p}\psi'-1|_\calA=|\alp_1|_\calA\le s^p$. It remains to show that $\lam'\prec py^p$ and $s(\lam')\le s^2$. Since $|\alp_2|_\calA\le s^p$ and $Q\in \calA^\circ$, we immediately obtain that $|\alpha^2 Q+\alp_2|_\calA\le s^2$. The desired bound on the second term of $\lam'$ is obtained by Lemma~\ref{autlem}(ii): $$|p^{-1}y^{-p}(\lam(g(y))-\lam(y))|_\calA\le |p^{-1}y^{-p}\lam|_\calA\cdot \|g\|=s^2.$$
\end{proof}

\subsubsection{Improving $\lam$ in case (2)}\label{case2sec}
Note that $c_n^{-1}\lam\in \calA$ possesses an $n$-th root $\rho(y)=c_n^{-1/n}\lam^{1/n}$ with dominant term $y$. So, $\rho$ is an element of $\Aut(\calA)^\circcirc$ and we define $g(y)$ to be the inverse automorphism.

\begin{lem}\label{lamlem2}
Assume that $\varphi=y^p+\lam$ is a simple two-term decomposition in case (2), and let $g(y)$ be as above. Then $\varphi'=\varphi(g(y))$ has a two-term decomposition $\psi'+\lam'$ such that $s'_p=s_p(\psi')\le s^p$ and $s'=s(\lam')\le\gamma s$, where $\gamma:=|t^{-1}py^{p}|_\calA<1$.
\end{lem}
\begin{proof}
Present $\lam$ as $c_ny^n(1+\lam_0)$. Then $|\lam_0|_\calA=s$ and since $(n,p)=1$ we also have that $|(1+\lam_0)^{1/n}-1|_\calA\le|\lam_0|_\calA\le s$ (for example, use that the binomial series $(1+u)^{1/n}=\sum_{i\in\bfN}\binom{1/n}{i}u^i$ has coefficients in $\kcirc$). Since $\rho=y(1+\lam_0)^{1/n}$, we obtain that $\|\rho\|\le s$, and hence $\|g\|\le s$ by Corollary \ref{invnormcor}(i).

Note that $\varphi'=g^p+\lam(g)=g^p+c_n\rho(g)^n=g^p+t$. Setting $g=y+\alpha y$, we obtain that $g^p=(1+p\alpha Q +\alpha^p)y^p$ for some $Q\in \calA^\circ$. Since $|\alp|_\calA=\|g\|\le s$, we have that $|\alpha^p|_\calA\le s^p$. Taking $\alpha^p=\alp_1+p\alp_2$ as in Lemma~\ref{palplem}, we obtain $\varphi'=\psi'+\lambda'$, where $\psi'(y^p)=y^p(1+\alp_1)$ and $\lambda'(y)=t+py^p(\alpha Q+\alp_2)$. It remains to observe that $s'_p=|\alp_1|_\calA\le |\alpha^p|_\calA=s^p$ and $$s'=|t^{-1}py^p(\alpha Q+\alp_2)|_\calA\le|t^{-1}py^{p}|_\calA\cdot |\alpha Q+\alp_2|_\calA\le\gamma s.$$
Finally, $\gamma<1$ by Lemma \ref{tametermlem}.
\end{proof}

\begin{proof}[Proof of Theorem~\ref{binomth}]
By Lemma \ref{psilem} $f$ possesses presentations $\varphi$ which admit simple two-term decompositions. So, without restriction of generality, we can assume that the initial decomposition of $\varphi_0$ is simple, say $\varphi_0=y^p+\lam_0$. We inductively define a sequence of presentations $\varphi_i$ with simple two-term decompositions $y^p+\lam_i$ as follows:

(i) $\varphi'_i=\varphi_i\circ g_{i}$, with $g_i$ as in \S\ref{case1sec} or \S\ref{case2sec}, and, depending on the case, its two-term decomposition $\psi'_i+\lam'_i$ is obtained from $\varphi_i$ via Lemma~\ref{lamlem1} or Lemma~\ref{lamlem2}, respectively.

(ii) $\varphi_{i+1}=h_{i}\circ\varphi'_{i}$ with $h_i$ as in \S\ref{hsec}, and its two-term decomposition $y^p+\lam_{i+1}$ is obtained from $\varphi'_{i}$ via Lemma~\ref{psilem}.

Set $s_i=s(\lam_i)$. Then $s_{i+1}\le s_i^2$ in case (1), $s_{i+1}\le \max(s_i^p,\gamma s_i)$ in case (2), and $s_p(\psi'_{i})\le s_i^p$ in both cases by Lemmas~\ref{psilem}, \ref{lamlem1} and \ref{lamlem2}. Since $\gamma$ only depends on the tame term, it is fixed in the process, and we obtain that the sequences $(s_i)$ and $(s_p(\psi'_{i}))$ are strictly decreasing and converge to zero. By the construction, $\|h_i\|=s_p(\psi'_{i})$ and $\|g_i\|=s_i$. By Corollary \ref{invnormcor} the limits $g=\lim_n g_0\circ \ldots\circ g_n$ and $h=\lim_n h_n\circ \ldots\circ h_0$ exist, and the limit two-term decomposition of $\varphi:=h\circ \varphi_0 \circ g$ satisfies $s=s_p=0$. So, $\varphi=y^p+t$ is a presentation of $f$ as required.
\end{proof}

\begin{rem}\label{binomrem}
(i) We worked with strictly analytic annuli, but Theorem~\ref{binomth} and its proof apply to non-strict annuli as well. The only difference is that one has to work with skeletons $l=[r_1,r_2]$ and one cannot normalize the coordinate so that the dominant term $c_py^p$ becomes monic.

(ii) In particular, Theorem~\ref{binomth} applies to annuli $A(r_1,r_2)$ over a trivially valued algebraically closed ground field $k$ of positive characteristic. The interesting case is obtained for $r_2<1$, when the theorem is equivalent to the following statement: if $K=k((t))$ and $L/K$ is a separable extension of degree $p$, then there exist uniformizers $x\in K$ and $y\in L$ such that $y^p+y^n=x$, where $n>p$ and $(p,n)=1$. We have used here that the coefficient of $y^n$ can be taken $1$ by Lemma \ref{tametermlem2}(i). In particular, since $\delta_{L/K}=n-1$, we obtain that the $k$-isomorphism class of the extension $L/K$ (up to isomorphisms of both $K$ and $L$) is determined by the different. To the best of our knowledge this is a new result.

(iii) Note also that if $L/K$ in (ii) is Galois, then it is an Artin-Schreier extension, say $L=K(y)$ with $y^p-y\in K$. Replacing $y$ by $y-a$ with $a\in K$ one achieves that $y^p-y=x^{-l}$ for a uniformizer $x\in K$ and $l\ge 1$ with $(p,l)=1$. A simple classical computation shows that $\delta_{L/K}=(l+1)(p-1)$, in particular, $\delta_{L/K}$ is divisible by $p-1$. Since $\delta_{L/K}\ge p$, (ii) implies that the latter property characterizes Galois extensions among all separable wild extensions $L/K$ of degree $p$. In addition, the extension in (ii) is Galois if and only if $(p-1)|(n-1)$, and then $l=(n-p)/(p-1)$ for appropriate Artin-Schreier parameters.

(iv) One can also prove (ii) directly using the method of Section~\ref{proofsec}. The argument simplifies in two aspects: there is only the binomial case, and the automorphisms of $K$ and $L$ are given by series $t+\sum_{i\ge 2}a_it^i$ without negative power terms. In particular, an initial presentation looks as $x=\psi(y^p)+\lam(y)$, where $y^p$ and $cy^n$ are the dominant terms of $\psi$ and $\lam$, respectively.
\end{rem}

\subsection{\'Etale $p$-covers of punctured discs}
A similar and slightly simpler classification exists for \'etale $p$-covers of punctured discs.

\subsubsection{Punctured and pointed discs}
By a {\em pointed disc} we mean a closed disc $D$ with a fixed $k$-point $O$ called the origin. Its skeleton $l$ is the interval $l=[O,q]$, where $q$ is the maximal point of $D$. A monic coordinate is any function $t$ on $D$ vanishing at $O$ and taking $D$ isomorphically onto the unit closed disc $\calM(k\{t\})$. An analytic space isomorphic to $D\setminus\{O\}$ will be called a {\em punctured disc}. By an {\em $m$-cover} of pointed or punctured discs, we mean a finite morphism $f\:Y\to X$ of degree $m$ between pointed or punctured discs. In particular, in the case of pointed discs, $f$ is totally ramified at the origin. In fact, we do not really have to distinguish the two notions:

\begin{lem}\label{punchpointed}
Any $m$-cover of punctured discs extends uniquely to an $m$-cover of pointed discs.
\end{lem}
\begin{proof}
Let $f\:Y\to X$ be an $m$-cover of punctured discs. Choosing monic coordinates we can assume that $Y=D_y\setminus\{O_y\}$ and $X=D_x\setminus\{O_x\}$ are the standard punctured unit discs with coordinates $y$ and $x$. In particular, $f$ is given by an invertible function on $Y$. Viewing $f$ as a function on a closed annulus $A(r,1)\subset D_y$ we can present it as a series $\sum_{i\in\bfZ}a_iy^i$, and the same series represents $f$ for any $r$ with $0<r\le 1$. By our assumption $a_my^m$ is the dominant term on any $A(r,1)$. It follows that all terms with $i<m$ vanish, in particular, $f$ extends to the finite map $D_y\to D_x$.
\end{proof}

\subsubsection{\'Etale $p$-covers}
An $m$-cover of pointed discs will be called {\em \'etale} if it is \'etale on the corresponding punctured discs. If $(m,p)=1$ then it is easy to see that \'etale $m$-covers $Y\to X$ are Kummer, and all \'etale $m$-covers of $X$ are $X$-isomorphic. \'Etale $p$-covers are classified as follows:

\begin{theor}\label{binomdiscsth}
Assume that $f\:Y\to X$ is an \'etale $p$-cover of pointed discs.

(1) If $k$ is of mixed characteristic then the cover is Kummer, that is, $y^p=x$ for appropriate monic coordinates.

(2) If $k$ is of equal characteristic then the cover is binomial. In fact, $f$ admits a monic presentation $x=y^p+c_ny^n$ if and only if $|c_ny^{n-p}|_l=\delta_f|_l$, where $l=l(Y)$ is the skeleton of $Y$.
\end{theor}
\begin{proof}
Choosing monic coordinates of $Y$ and $X$ we obtain a presentation $x=\varphi(y)$ of $f$. Clearly, $c_py^p$ is the dominant term of $\varphi$, and by a linear change of coordinates we can assume that $c_p=1$. Since $f$ is totally ramified at the origin $O\in Y$, one in fact has that $\varphi=y^p+\sum_{i>p}c_iy^i$. The fact that $f$ is \'etale outside of $O$ means that $\varphi'$ is invertible outside of $O$, and hence the dominant term of $\varphi'$ is the first non-zero term. Now we have two cases:

(1) If $\cha(k)=0$ then $py^{p-1}$ is the dominant term of $\varphi'$. In particular, $|c_i|<|p|$ for any $i\in\bfN\setminus p\bfN$ and $\delta_l:=\delta_f|_l$ identically equals $|p|$.

(2) If $\cha(k)=p$ then the dominant term of $\varphi'$ is $nc_ny^{n-1}$, where $n$ is the smallest number such that $(n,p)=1$ and $c_n\neq 0$. In this case, we call $c_ny^n$ the {\em dominant tame term} of $\varphi$. Note that $|y^{-p}t|_l=\delta_l$.

We see that in the case of pointed discs the separation into two cases goes accordingly to the characteristic of $k$. The remaining argument follows the proof of Theorem~\ref{binomth} and is slightly simpler, so we only outline it. By a two-term decomposition of a monic presentation $\varphi$ we mean a decomposition $\varphi=\psi(y^p)+\lam(y)$, such that $y^p$ is the dominant term of $\psi$, $|\lam|<|p|$ in case (1), and $c_ny^n$ is the dominant term of $\lam$ in case (2). Then one modifies $\varphi$ by composing with automorphisms of punctured discs, precisely as in the proof of Theorem~\ref{binomth}. The computations simplify because these automorphisms are given by series $t+\sum_{i>2}a_it^i$ without negative terms.
\end{proof}

\section{Skeletons}\label{skelsec}

\subsection{Skeletons of nice curves}

\subsubsection{Four levels of skeletons}
In the simplest form, skeletons in Berkovich geometry are certain nice topological subspaces of an analytic space $X$. Usually, they admit a few natural levels of enhancements. Namely, the basic topological level (Top) can be enhanced to a reduction level (Red), a tropical or combinatorial level (Trop), and what we call a log reduction level (LogRed). There are forgetful functors (LogRed)$\to$(Red)$\to$(Top) and (LogRed)$\to$(Trop)$\to$(Top), and (LogRed) is obtained by combining the information of (Red) and (Trop). So, loosely speaking the information kept on these levels fits into a ``bicartesian diagram''
$$
\xymatrix{
(Top)\ar@{^(->}[d]\ar@{^(->}[r] & (Trop) \ar@{^(->}[d]\\
(Red)\ar@{^(->}[r] & (LogRed)
}
$$

We do not try to formalize the above principle, but we indicate the worlds these objects live at, and we will later see this in detail in the case of curves.

(Top) A skeleton $\Gamma_X$ is a topological space.

(Red) The reduction level corresponds to algebraic geometry over $\tilk$. A typical example, is the reduction $\gtX_s$ of $X$ corresponding to a formal model $\gtX$, and $\Gamma_X$ is reconstructed as the topological realization of a simplicial space related to $\gtX_s$.

(Trop) The tropical level corresponds to tropical geometry or PL geometry over $|k^\times|$. A typical example is the natural metric structure on skeletons of curves.

(LogRed) The log reduction level is an amalgam of (Red) and (Trop). Sometimes this can be done ad hoc, for example, via metrized curve complexes of Amini-Baker (\cite{Amini-Baker}). A more conceptual way to treat this information is to work within the framework of log geometry over $\Spec(\tilk)$ provided with the log structure $(\kcirc\setminus\{0\})/(1+\kcirccirc)\to\tilk$. The latter is associated with the prelog structure $\kcirc\setminus\{0\}\to\tilk$. Non-canonically, it is also associated with $|k^\times|_{\le 1}\to 0$.

\subsubsection{Triangulations}
Following Ducros, by a {\em triangulation} of a nice $k$-analytic curve $X$ we mean a finite set $V=V^{(1)}\coprod V^{(2)}\subset X$ of points of types 1 and 2 such that $V^{(2)}\neq\emptyset$ and $X\setminus V$ is a disjoint union of open discs, and finitely many punctured open discs and open annuli. One can enhance $V$ to the four levels as follows.

\subsubsection{Topological skeletons of curves}
Take $V=V_\Gamma$ to be the set of vertices of the graph $\Gamma$ whose edges are skeletons of punctured disc and annuli components of $\Gamma\setminus V$. Then $\Gamma$ is a (topological) skeleton of $X$ in the sense of \cite[Section 3.5]{CTT}. We call elements of $V^{(1)}=\Gamma^{(1)}$ {\em marked points of $X$} of {\em infinite vertices} of $\Gamma$, and we call elements of $V^{(2)}$ {\em finite} or {\em ordinary vertices}.

\subsubsection{Tropical skeletons of curves}
The graph $\Gamma$ inherits a natural metric from $X$ which is singular at all marked points. In addition, one often provides $\Gamma$ with the genus function $g\:\Gamma^0\to\bfN$ (\cite[\S3.3.2]{CTT}) making it a metric genus graph in the sense of \cite{CTT}.

\subsubsection{The reduction level}
Providing a triangulation $V$ is equivalent to marking a finite set of type 1 points $D=\Gamma^{(1)}$ and providing a semistable model $\gtX$ of $X$ such that the reduction map $\pi_\gtX\:X\to\gtX_s$ maps $D$ bijectively onto a finite set $D_s$ of smooth points of the nodal $\tilk$-curve $C=\gtX_s$. The formal fiber $X_q=\pi_\gtX^{-1}(q)$ over $q\in C$ can be as follows:

(1) If $q$ is a generic point then $X_q=\{x\}$ for $x\in V^{(2)}$, and $C_x$ is the corresponding normalized component of $C$.

(2) If $q$ is a nodal point then $X_q$ is an open annulus whose skeleton is a finite edge of $\Gamma$.

(3) If $q\in D_s$ is a marked point then $X_q$ is a pointed disc with marked point $Q\in D$ lifting $q$. Its skeleton is an infinite edge of $\Gamma$.

(4) If $q$ is an ordinary smooth point then $X_q$ is an open disc component of $X\setminus V$.

It follows that the multipointed $\tilk$-curve $(C,D_s)$ can be constructed from $X$ and $\Gamma$ in the following ad hoc manner. First, one takes the smooth curve $\tilC=\coprod_{x\in V^{(2)}} C_x$. Second, for any finite edge $e=[x_1,x_2]$ in $\Gamma$, one identifies the points of $C_{x_i}$ corresponding to $e$
(it may happen that $x_1=x_2$). This pushout procedure outputs a nodal pinching $C$ of $\tilC$. Third and final, for any edge $e=[x,d]$ with $d\in\Gamma^{(1)}$, one marks the point on $C_x$ corresponding to $e$.








\subsubsection{The log reduction level}\label{logredsec}
This level combines the tropical and reduction level. It can be achieved by enriching the metric graph $\Gamma$ by the relevant algebra-geometric information. This is done in the definition of metrized curve complexes in \cite{Amini-Baker}. Loosely speaking, at each finite vertex $v$ one installs the $\tilk$-curve $C_v$ and associates the edges starting at $v$ to points of $C_v$.

In this paper we will use another approach, which we find more conceptual. Note, nevertheless, that for semistable curves both approaches are equivalent, so everything can be translated to the language of \cite{ABBR}. We will enrich $C$ by a log structure $M_C$, which contains the information about the metric structure of $\Gamma$. Loosely speaking this provides a way to remember $|\pi|$ after reducing an equation of the form $xy-\pi=0$ modulo $\kcirccirc$. In fact, such situation was the original motivation for introducing log structures by Fontaine-Illusie.

The following ad hoc definition will suffice for our needs in the paper. Let $s^\rmlog$ denote the log enrichment of $s=\Spec(\tilk)$ by the log structure $M_s\to\tilk$ associated with $\kcirc\setminus\{0\}\to\tilk$. Note that $M_s=(\kcirc\setminus\{0\})/(1+\kcirccirc)$, and non-canonically this is also the log structure associated with $\oM_s=|k^\times|_{\le 1}\to 0$. We enrich the nodal curve $C$ to an $s^\rmlog$-curve $(C,M_C)$ as follows:

(1) The log structure is $s^\rmlog$-trivial at any point $q$ which is not marked or nodal, that is, $\oM_q=\oM_s$.

(2) The log structure at a marked point $u\in D_s$ is generated by a uniformizer $x_u\in m_u$, that is, $\oM_u=\oM_s\times x_u^\bfN$.

(3) The log structure at a nodal point $z\in D_s$ is generated by uniformizers $x_1$ and $x_2$ of the two branches at $z$ modulo a relation $x_1x_2=\pi$, where $\pi\in\kcirccirc$ is such that $|\pi|$ is the exponential modulus of the formal fiber $X_z$. Namely, $\oM_z=\oM_s\times x_1^\bfN\times x_2^\bfN/(x_1x_2=|\pi|)$. Note that this log structure is Zariski when $x_i$ lie on different components of $C$, but it is only an \'etale log structure otherwise, since $\oM_C$ only makes sense in the \'etale topology, and $x_1,x_2$ are only defined \'etale-locally.

\begin{rem}
We chose a relatively ad hoc definition of $M_C$, which was also used in the proof of \cite[Theorem~4.4]{logPicard}. A more conceptual way is outlined in appendix \ref{logredapp}: one promotes all objects, including $X$, $\gtX$ and $C$, to the log geometric level.
\end{rem}

\subsubsection{Nice $s^\rmlog$-curves}
By a {\em nice $s^\rmlog$-curve} we mean a nodal $\tilk$-curve $C$ with a finite set $U$ of smooth marked points provided with a following log structure: $\oM_q=\oM_s$ at ordinary smooth points, $\oM_u=\oM_s\times x_u^\bfN$ at marked points, and $\oM_z=\oM_s\times x_1^\bfN\times x_2^\bfN/(x_1x_2=r)$ at nodal points $z$, where $r\in|k^\times|_{<1}$. Note that $r$ is determined by the monoid $\oM_z$, and we will call $r=r(z)$ the {\em modulus} of the log node $z$. So, it is obvious that $(C,M_C)$ determines the metric skeleton $\Gamma$.

\begin{rem}
Any nice $s^\rmlog$-curve $(C,M_C)$ is log smooth over $s^\rmlog$. This fact is not essential for this paper, but we think it is a strong indication in favor of using the log reduction language.
\end{rem}

\subsection{Skeletons of morphisms}
Next, we discuss skeletons of finite covers of curves.

\subsubsection{Finite morphisms of $s^\rmlog$-curves}
We say that a morphism $\lam\:(C,M_C)\to(D,M_D)$ of $s^\rmlog$-curves is {\em finite} if the morphism $C\to D$ is. For any smooth point $q\in C$ we define $n_q$ to be the usual multiplicity of $q$ in the fiber. For a nodal point $z$ a priori there are two multiplicities corresponding to the two branches. However, the log structure forces them to coincide, and we will freely use the notation $n_z$. Let us justify this claim.

\begin{lem}\label{nlem}
Let $\lam$ be as above and $z\in C$ a nodal point with preimages $z_1,z_2$ in the normalization $C^\nor$. Then the multiplicities $n_i=n_{z_i}$ are equal and the moduli of $z$ and $t=\lam(z)$ are related by $r(z)^{n_1}=r(t)$.
\end{lem}
\begin{proof}
If $x_i$ are uniformizers at $z_i$ and $y_i$ are uniformizers at the images of $z_i$ in $D^\nor$, then $(y_i)=(x_i^{n_{i}})$. So, $x_1^{n_{1}}x_2^{n_2}=y_1y_2=r(t)$ in $\oM_t$. It follows easily from the description of $\oM_t$ that $n_1=n_2$, and then also $r(z)^{n_1}=x_1^{n_{1}}x_2^{n_1}=r(t)$.
\end{proof}

\begin{rem}
A finite morphism $\lam\:(C,M_C)\to(D,M_D)$ of nice $s^\rmlog$-curves is log smooth at a point $q\in C$ if and only if $(n_q,p)=1$. Again, we will not really need this, but this provides a conceptual explanation to the fact that in all our work the tame case is very simple.
\end{rem}

\subsubsection{Compatible triangulations}
Let $f\:Y\to X$ be a finite morphism of nice curves. A pair of triangulations $V_Y\subset Y$ and $V_X\subset X$ is called {\em compatible} if $f^{-1}(V_X)=V_Y$. Such a pair extends to all four levels, in particular, it induces a finite morphism of the associated log reductions $\lam\:(C,M_C)\to(D,M_D)$. Clearly, using $\lam$ one can also descend to lower levels and obtain a finite morphism of metric graphs $\Gamma_f\:\Gamma_Y\to\Gamma_X$ and a finite morphism of nodal $\tilk$-curves $h\:C\to D$. Recall that the multiplicity function $n_f\:Y\to\bfN_{\ge 1}$ of $f$ (e.g., see \cite[\S2.1.5]{radialization}) is constant on each edge $e$ of $\Gamma_Y$ by \cite[Lemma~3.5.10]{CTT}. Clearly, $n_e=n_z$ for the corresponding nodal point $z\in C$, and $f|_e$ is monomial of degree $n_e$.

\begin{rem}\label{logfavorrem}
The maps $h$ and $\Gamma_f$ satisfy certain natural restrictions. For example, the multiplicities of $h$ along two branches of a nodal point coincide, and $\Gamma_f$ satisfies a harmonicity condition from \cite{ABBR}. We leave spelling this out to the interested reader. On the log reduction level, one only has to assume that $\lam$ is finite, and the other conditions follow. This is one more argument in favor of the logarithmic language in our setting.
\end{rem}

\subsubsection{Triangulation of covers}\label{logredsec2}
By a {\em triangulation} of a finite generically \'etale morphism $f\:Y\to X$ of nice curves we mean compatible triangulations $V=(V_Y,V_X)$ such that $V_Y$ contains the ramification set $\Ram(f)$. In this case, $f$ splits on the complement of the triangulations to a disjoint union of annular, disc, and punctured disc \'etale covers. The associated morphism $\lam\:(C,M_C)\to(D,M_D)$ of log reductions will be called the {\em log reduction} of $f$ associated with $V$.

\subsection{$p$-enhancements}

\subsubsection{The tame case}
It seems that if $f\:Y\to X$ is residually tame then any log reduction catches all discrete and $\tilk$-geometric information about $f$. We do not try to find a rigorous formulation of this principle, but illustrate it with the following example. The associated tropical skeleton $\Gamma_f\:\Gamma_Y\to\Gamma_X$ satisfies the Riemann-Hurwitz formulas at all non-boundary vertices $v$ of $\Gamma_Y$ of type 2:
$$2g(v)-2=n_v(2g(f(v))-2)+\sum_{e\in\Br(v)}(n_e-1).$$
Indeed, this is the usual RH formula for the reduction, but all entries are encoded already in $\Gamma_f$. For example, $n_e$ is the degree of $f$ on $e$.

\subsubsection{The different}
A new discrete invariant, which is non-trivial when $f$ is residually wild, was introduced in \cite{CTT}. It is the restriction of the different function $\delta_f$ onto $\Gamma_Y$. As we saw, the different is monomial on edges, and using its slopes one restores the RH formula and even extends it to the case when $f$ is not residually tame at $v$:
$$2g(v)-2=n_v(2g(f(v))-2)+\sum_{e\in\Br(v)}(-\slope_e(\delta_f)+n_e-1).$$

\subsubsection{The $p$-enhancement}
The bivariant forms $\tiltau_{f,y}$ we introduced earlier provide a natural way to upgrade $\delta_f|_{\Gamma_Y}$ to the log reduction level. We will only consider a special case when the wildness is not too large. Let $\lam\:C\to D$ be a finite morphism of nice $s^\rmlog$-curves, let $\{C_i\}_{i\in I}$ be the set of normalized irreducible components of $C$, and let $\eta_C=\{\eta_i\}_{i\in I}$ be the set of generic points of $C$. Assume that any closed point $v\in C$ at which $\lam$ is not log-\'etale satisfies $n_v=p$. In particular, if a morphism $C_i\to D_j$ is inseparable, then it is a geometric Frobenius. Then by a {\em $p$-enhancement} of $\lam$ we mean a following data:
\begin{itemize}
\item[(1)] meromorphic bivariant forms $\phi_i\in\omega_{k(\eta_i)/k(\lam(\eta_i))}$,

\item[(2)] a function $\delta\:\eta_C\to(0,1]$,
\end{itemize}
such that the following four conditions are satisfied, where for any point $q\in C_i$ we set $s_q=-\logord_q(\phi_i)$.
\begin{itemize}
\item[(1)] If $k(\eta_i)/k(\lam(\eta_i))$ is separable then $\delta(\eta_i)=1$ and $\phi_i=\tau_{k(\eta_i)/k(\lam(\eta_i))}$.

\item[(2)] If $k(\eta_i)/k(\lam(\eta_i))$ is inseparable then either $|p|<\delta(\eta_i)<1$ and $\phi_i$ is exact, or $\delta(\eta_i)=|p|$ and $\phi_i$ is mixed.

\item[(3)] Let $\delta_\Gamma\:\Gamma_C\to[0,1]$ be the extension of $\delta$ which is monomial on any edge and satisfies the condition $\delta_\Gamma(u)=|n_u|$ for any marked point $u\in C$. Then for any oriented edge $e$ starting at a finite vertex $q\in C_i$ the degree of $\delta_\Gamma$ on $e$ equals $s_q$.

\item[(4)] If $q$ sits over an ordinary smooth point then $s_q=0$.
\end{itemize}

\begin{rem}
Condition (3) can be reformulated more explicitly as the combination of the following two conditions:

(i) For any nodal point $z\in C$ and any order of the preimages $z_l\in C_{i_l}$, $l=1,2$ one has that $\delta(\eta_{i_1})r(z)^{s_{z_1}}=\delta(\eta_{i_2})$. In particular, $s_{z_1}=-s_{z_2}$.

(ii) For any marked point $u\in C_i$, if $\cha(k)>0$ then $s_q<0$ whenever $u$ is wild and $s_q=0$ whenever $u$ is tame, and if $\cha(k)=0$ then $s_q=0$, $\delta(\eta_i)=|n_u|$.


\end{rem}

\begin{rem}
It is an interesting question if there is a conceptual way to interpret our ad hoc definition. In particular, can it be related to a generalization of dualizing sheaves to morphisms of log schemes?
\end{rem}

\subsubsection{Enhanced log reduction of covers}
Now, let us return to our study of reductions of $f$.

\begin{theor}\label{logredth}
Let $f\:Y\to X$ be a generically \'etale, minimally wild on $Y$, finite morphism of nice $k$-analytic curves, let $V$ be a triangulation of $f$, and let $\lam\:C\to D$ be the corresponding log reduction of $f$. Then the forms $\tiltau_{f,y}$ for $y\in V^{(2)}$ and the restriction $\delta\:V^{(2)}\to(0,1]$ of $\delta_f$ provide a $p$-enhancement of $\lam$.
\end{theor}
\begin{proof}
In view of Lemma~\ref{tiltaulem}(ii), the forms $\tiltau_{f,y}$ satisfy condition (1) by Lemma~\ref{tautamelem}, and they satisfy condition (2) by Theorem~\ref{degpth}. Consider the function $\delta_\Gamma=\delta_f|_{\Gamma_C}$. We claim that it satisfies condition (3). Indeed, the condition for marked points is satisfied by Theorem~\ref{maindif}(i), and the condition for slopes is satisfied by Lemma~\ref{tiltaulem}(i). Finally, condition (4) follows from \cite[Theorem~6.1.9(i)]{CTT}.
\end{proof}

\section{The lifting theorem}\label{mainsec}

\subsection{Star-shaped curves}
By a {\em star-shaped} curve we mean a pair $(X,x)$, where $X$ is a nice $k$-analytic curve and $x\in X$ is a point of type 2 such that $X\setminus\{x\}$ is a disjoint union of open discs and semi-open annuli. A morphism $f\:(Y,y)\to (X,x)$ is a morphism $f\:Y\to X$ such that $f^{-1}(x)=\{y\}$.

\begin{lem}\label{starlem}
Assume given a finite separable extension $L/K$ of one-dimensional analytic $k$-fields of type 2.

(i) There exists a boundaryless star-shaped curve $(X_0,x)$ with an isomorphism $\calH(x)=K$ such that $x\in\Int(X_0)$.

(ii) For any curve $(X_0,x)$ as in (i) there exists a star-shaped subdomain $(X,x)\into(X_0,x)$ such that $x\in\Int(X)$ and $L/K$ lifts to a morphism $f\:(Y,y)\to(X,x)$ of star-shaped curves (in the sense that $\calH(y)=L$ as $K$-fields) such that $f$ is finite and \'etale.
\end{lem}
\begin{proof}
This is a simple consequence of the following two facts: (1) any point of type 2 has a fundamental family of star-shaped neighborhoods by \cite[Proposition~3.6.1]{berbook}, (2) given a germ $X_x$ of an analytic space $X$ at a point $x$, the category of its \'etale covers by germs $Y_y$ is equivalent to the category of finite separable extensions of $\calH(x)$, and the equivalence is obtained by sending $Y_y$ to $\calH(y)/\calH(x)$.
\end{proof}




\subsection{The main theorem}

\subsubsection{Minimally wild morphisms}
We say that a finite \'etale morphism $\lam\:C\to D$ of nice $s^\rmlog$-curves is {\em minimally wild} if any point $v\in C$ such that $f$ is not log-\'etale at $v$ satisfies $n_v=p$, and any fiber $f^{-1}(u)$ contains at most one such point. In particular, it follows that for any irreducible component $D_i$ there exists at most one irreducible component $C_j$ above $D_i$ such that $\lam_j\:C_j\to D_i$ is not generically \'etale, and in this case $\lam_j$ is the geometric Frobenius.

\subsubsection{The lifting theorem}
Here is the main lifting result of the paper.

\begin{theor}\label{mainth}
Let $\lam\:C\to D$ be a minimally wild morphism of nice $s^\rmlog$-curves provided with a $p$-enhancement $(\phi,\delta)$. Then there exists a morphism of nice $k$-analytic curves $f\:Y\to X$ such that $\lam$ is a logarithmic reduction of $f$.
\end{theor}
\begin{proof}

Step 1. {\it Setup.} We will use the following notation:

(1) $I$ and $J$ are the sets of normalized irreducible components $D_i$ of $D$ and $C_j$ of $C$. The generic points will be denoted $x_i\in D_i$ and $y_j\in C_j$, and the same letters will be used for the corresponding points of type 2 in $k$-analytic curves we will construct.

(2) The sets of marked points will be denoted $U\subset D$ and $V\subset C$. Note that we are allowed to enlarge the logarithmic structures on $D$ and $C$ in a compatible way. On the level of marked points this simply means that we add few smooth points to $U$ and add their preimages to $V$. This operation does not affect the $p$-enhancement. 

(3) The sets of nodal points will be denoted $Z\subset D$ and $T\subset C$.

Step 2. {\it Construction of star-shaped curves.} By Lemma~\ref{starlem}(i), for each $i\in I$ we can choose a star-shaped curve $(X_i,x_i)$ such that $x_i\in\Int(X_i)$ and $K_i:=\calH(x_i)$ satisfies $\tilK_i=k(D_i)$. The connected components of $X_i\setminus\{x_i\}$ are parameterized by the closed points of the smooth compactification of $D_i$, and we remove from $X_i$ the components parameterized by the points not lying in $D_i$. Then $(X_i,x_i)$ becomes a star-shaped curve such that whose reduction $C_{x_i}$ at $x_i$ is $D_i$.

For each $C_j$ mapped to $D_i$ we define $L_j/K_i$ as follows. If $C_j\to D_i$ is generically \'etale then $L_j/K_i$ is the unramified extension lifting $k(C_j)/k(D_i)$. If $C_j\to D_i$ is radicial then by Theorem~\ref{type2fields} there exists an extension $L_j/K_i$ such that $\delta_{L_j/K_i}=\delta(y_j)$ and $\tiltau_{L_j/K_i}=\phi_j$. By Lemma \ref{starlem}(ii), shrinking $X_i$ if necessary we can also construct for each $j$ a star-shaped curve $(Y_j,y_j)$ with a finite morphism $f_j\:Y_j\to X_i$ lifting the extension $L_j/K_i$.

For a point $q\in D_i$ the corresponding component $X_{i,q}$ of $X_i\setminus\{x_i\}$ is either a disc or an annulus. Shrinking $X_i$ and replacing $Y_i$ by its preimage we can achieve that $X_{i,q}$ is an annulus whenever $q$ is not an ordinary smooth point of $D$. Furthermore, enlarging the log structure we can also achieve that $X_{i,q}$ is an annulus if and only if $q$ is not an ordinary smooth point of $D$. Finally, shrinking the annuli $X_{i,q}$ we can assume that they all are of the same radius $r\in|k^\times|$ such that $r^2>r(z)$ for any $z\in Z$, and we fix a number $s\in|k^\times|$ such that $r<s<1$.

Step 3. {\it Objects of gluing data.} The morphism $f\:Y\to X$ will be glued from morphisms of three types:
\begin{itemize}
\item[(1)] $f_i\:Y_i=\coprod_{j\to i}Y_j\to X_i$ for $i\in I$, where the morphism $f_j\:(Y_j,y_j)\to(X_i,x_i)$ were constructed in Step 2.
\item[(2)] $f_u\:Y_u=\coprod_{v\to u}Y_v\to X_u$ for $u\in U$, where $(X_u,x_u)$ is a pointed closed disc of radius $s$ with coordinate $x$, and each $f_v\:(Y_v,y_v)\to(X_u,x_u)$ is an $n_v$-cover as follows. Let $C_i$ be the component containing $v$, $l=\logord_v(\phi_i)$ the slope of $\delta$ on $e_v$, and $c\in k$ such that $|c|=\delta(y_i)$. Then $f_v$ is the Kummer cover given by $x=y^{n_v}$ if $l=0$, and $f_v$ is the binomial cover given by $x=y^p+cy^l$ if $l\neq 0$.
\item[(3)] $f_z\:Y_z=\coprod_{t\to z}Y_t\to X_z$ for $z\in Z$, where $X_z$ is a closed annulus of radii $s^{-1}r(z)$ and $s$ with coordinate $x$, and $f_t\:Y_t\to X_z$ is an $n_z$-cover as follows. Let $C_{j_1}$ and $C_{j_2}$ be the components containing the preimages of $t$ in $C^\nor$ (possibly, $j_1=j_2$), and let $c\in k$ and $l\in\bfZ$ be such that $|c|=\delta(y_{j_1})$ and $|c|r(t)^l=\delta(y_{j_2})$. Then $f_t$ is the Kummer cover given by $x=y^{n_t}$ if $l=0$, and $f_t$ is the binomial cover given by $x=y^p+cy^l$ if $l\neq 0$.
\end{itemize}

Step 4. {\it The gluing maps.} To obtain $X$ we will patch the star-shaped curves $X_i$ by discs $X_u$ and glue these curves along the annuli $X_z$. This will be done so that the gluings extend to the morphisms $f_i$, $f_u$ and $f_z$. We have two cases:

(1) For each $u\in U$, consider the disc $X_u$ of radius $s$ and let $X'_u\subset X_u$ be the annulus of radii $r$ and $s$. Let $D_i$ be the normalized component containing $u$. The connected component of $X_i\setminus\{x_i\}$ corresponding to $u$ is a semi-open annulus $X_{i,u}$ of radii $r$ and $1$, and we define $X'_{i,u}$ to be the subannulus of radii $r$ and $s$. The base change of $f_v$ with respect to $X'_u\into X_u$ will be denoted $f'_v\:Y'_v\to X'_u$. Similarly, the base change of $f_{j,v}\:Y_{j,v}\to X_{i,u}$ with respect to $X'_{i,u}\into X_{i,u}$ will be denoted $f'_{j,v}\:Y'_{j,v}\to X'_{i,u}$.

We will glue $X_u$ and $X_i$ via an isomorphism $g_u\:X'_{u}\toisom X'_{i,u}$. If $\lam$ is log-\'etale over $u$, we choose an arbitrary $g_u$. Otherwise, there exists a single point $v\in\lam^{-1}(u)$ such that $n_v=p$. Let $C_j$ be the component of $v$, then $f'_{j,v}$ is an \'etale annular $p$-cover, and by Corollary~\ref{classcor} its isomorphism class is determined by the skeleton of the different. In Step 3, we chose the binomial or Kummer cover $f_v$ so that the skeletons of the differents of $f'_v$ and $f'_{j,v}$ are equal. Thus, $f'_v$ and $f'_{j,v}$ are isomorphic, and choosing appropriate isomorphisms $g_u$ and $g_v\:Y'_{j,v}\toisom Y'_v$ we obtain a gluing of morphisms $f_i$ and $f_v$ along $f'_v\toisom f'_{j,v}$.

Once $g_u$ is fixed we still have to glue other preimages of $X_u$. Namely, for any point $w\in\lam^{-1}(u)$ with $(n_w,p)=1$ we should lift the isomorphism $g_u$ to an isomorphism $g_w\:Y'_{j,w}\toisom Y'_w$. This is possible because all \'etale $n_v$-covers of $X'_u$ are $X'_u$-isomorphic by \S\ref{Kumsec}. This completes the gluing of $f_u$ and $f_i$ along $f'_u\toisom f'_{i,u}$.

(2) Let $z\in Z$ be a nodal point. The subannuli $A_1=A(r,s)$ and $A_2=A(s^{-1}r(z),r^{-1}r(z))$ are disjoint because $rs>r^2>r(z)$. Let $\tilz_1\in D_{i_1}$ and $\tilz_2\in D_{i_2}$ be the two preimages of $z$. We will glue $X_z$ to $X_{i_1,\tilz_1}$ and $X_{i_2,\tilz_2}$ along $A_1$ and $A_2$, respectively. This is done in a symmetric way, so we will only consider the case of $l=1$ and set $\tilz=\tilz_1$, $i=i_1$ and $X'_\tilz=A_1$ for shortness. The latter will be identified by an isomorphism $g_\tilz$ with the subannulus $X'_{i,\tilz}$ of $X_{i,\tilz}$ of radii $r$ and $s$.

Construction of $g_\tilz$ copies the construction of $g_u$ from the above case (1) almost verbatim. If the normalized morphism $C^\nor\to D^\nor$ is log \'etale over $\tilz$ then any $g_\tilz$ works. Otherwise, there is a single point $\tilt$ over $\tilz$ such that $n_\tilt=p$, and, copying the notation of (1), the $p$-covers $f'_\tilt$ and $f'_{j,\tilt}$ are isomorphic because they have the same skeletal different. We simultaneously choose gluing isomorphisms $g_\tilz\:X'_{\tilz}\toisom X'_{i,\tilz}$ and $g_\tilt\:Y'_{\tilt}\toisom Y'_{j,\tilt}$ establishing an isomorphism of the $p$-covers. For any other point $\tilw$ over $\tilz$, the morphisms $Y'_\tilw\to X'_\tilz$ and $Y'_{j,\tilw}\to X'_{i,\tilz}$ are Kummer covers of degree $n_\tilw\in\tilk^\times$, hence an isomorphism $g_\tilw\:Y'_\tilw\toisom Y'_{j,\tilw}$ can be chosen so that it lifts the already fixed isomorphism $g_\tilz$.

Step 5. {\it Verification.} It is easy to see that the connected components of $X\setminus\{x_i\}_{i\in I}$ are of three types:

(1) Open discs $D_q$ parameterized by ordinary smooth points $q\in D$. These are the disc components of the curves $X_i\setminus\{x_i\}$.

(2) Pointed open discs $(D_u,x_u)$ parameterized by points $u\in U$. Each $D_u$ is glued from the corresponding annulus $X_{i,u}$ and pointed disc $(X_u,x_u)$.

(3) Open annuli $A_z$ of radii $r(z)$ and $1$ parameterized by nodal points $z\in Z$. Each $A_z$ is glued from the corresponding annuli $X_{i_1,\tilz_1}$, $X_z$ and $X_{i_2,\tilz_2}$.

It follows that the points $\{x_i,x_u\}_{i\in I,u\in U}$ provide a triangulation with log reduction $C$. In the same way, one checks that the preimages of these points in $Y$ provide a triangulation of $Y$ and $\lam\:C\to D$ is the corresponding log reduction of $f$. Finally, by the construction of Step 2, the values of $\delta_f$ and  $\tiltau_f$ at the points $y_i$ coincide with $\delta(y_i)$ and $\phi_i$. So, $f$ is a required lift of $(\lam,\phi,\delta)$.
\end{proof}

\appendix

\section{The different function}\label{diffunsec}
Recall that a different function $\delta_f\:Y\to[0,1]$ was introduced in \cite{CTT}. It is $|k^\times|$-pm and assigns to $y\in Y^\hyp$ the value of the different of $\calH(y)/\calH(f(y))$.

\subsection{Restrictions on the slopes}
First, let us recall basic properties of $\delta_f$, see \cite[Corollary~4.1.8 and Theorems~4.2.6]{CTT}.

\begin{theor}\label{slopeth}
If $f\:Y\to X$ and $\Gamma_f\:\Gamma_Y\to\Gamma_X$ are as in \S\ref{convsec}, then

(i) $\delta_f$ is a $|k^\times|$-pm function on $Y^\hyp$, and it is monomial on each edge of $\Gamma_Y$,

(ii) $|n_f(y)|\le\delta_f(y)\le 1$ for any $y\in Y^\hyp$, and $\delta_f(y)=1$ if $f$ is residually tame at $y$,

(iii) for any $y\in Y^{(2)}$ and $v\in C_y$ the slope $s=\slope_v(\delta_f)$ satisfies $$|n_f(v)|\le\delta_f(y)\le|n_f(v)+s|.$$
\end{theor}

\begin{rem}
In fact, (ii) and (iii) are the only restrictions on $\delta$ and $s$, and any allowed combination can be already obtained for a binomial \'etale annular covering $A_1\to A_2$ of the form $x=t^n+ct^m$, with $n=n_v$, $s=n-m$ and $\delta_f(y)=|c|$.
\end{rem}

\begin{rem}\label{sloperem}
The conditions (ii) and (iii) of the theorem are most restrictive in the mixed characteristic case. For example, if $n=p$ then $|p|\le\delta_f\le 1$, $s\notin p\bfZ\setminus\{0\}$, and $s=0$ can happen only when $\delta_f(y)=1$ or $\delta_f(y)=|p|$.
\end{rem}

\subsection{Local behaviour of the slopes}
The following result describes the restrictions $\delta_f$ satisfies locally at a point $y\in Y$. It summarizes \cite[Theorems~4.5.4, 4.6.4, 6.1.9]{CTT}.

\begin{theor}\label{maindif}
Let $f\:Y\to X$ and $\Gamma_f$ be as in \S\ref{convsec}, then

(i) For any point $y\in Y^{(1)}$ of type 1 one has that $\slope_y(\delta_f)=\delta^{\rmlog}_{\calO_y/\calO_{f(y)}}$ and $\delta_f(y)=|n_y|$.

(ii) For a non-boundary point $y\in Y$ of type 2 with $x=f(y)$ one has that $$2g(y)-2-n_y(2g(x)-2)=\sum_{v\in C_y}(-\slope_v\delta_f+n_v-1).$$ In particular, almost all slopes of $\delta_f$ at $y$ equal $n_y^i-1$, where $n_y^i$ is the inseparability degree of $\wHy/\wHx$.

(iii) The different behaves trivially outside of $\Gamma_Y$ in the sense that $\slope_v\delta_f=n_v-1$ for any direction $v$ not pointing towards $\Gamma_Y$.
\end{theor}

\begin{rem}
(i) We call part (ii) of the theorem local Riemann-Hurwitz formula. It reduces to the Riemann-Hurwitz formula for $\tilf_y\:C_y\to C_x$ when $f$ is residually tame at $y$, but contains information not encoded in $\tilf_y$ otherwise. We show in this paper that this information is related to the reduction of $\tau_f$.

(ii) The theorem implies the global Riemann-Hurwitz formula for $f$, which includes correction terms at the boundary points when $X$ and $Y$ are not proper, see \cite[Theorems 6.2.3 and 6.2.7]{CTT}.
\end{rem}

\section{Log reduction of nice curves}\label{logredapp}

\subsection{Nice $k$-analytic log curves}
By a {\em multipointed} nice $k$-analytic curve we mean a nice $k$-analytic curve $X$ provided with a finite set $D$ of {\em marked} points of type 1. An equivalent way to encode this datum is to consider the associated {\em nice log curve} $(X,M_X)$ with the logarithmic structure $M_X\into\calO_X$ induced by $D$, that is, $M_{X,x}=\calO_{X,x}\setminus\{0\}$ and $\oM_{X,x}=\bfN$ if $x\in D$, and $M_{X,x}=\calO_{X,x}^\times$, $\oM_{X,x}=1$ otherwise. Note that $(X,M_X)$ is log smooth over $k$ with the trivial log structure $k^\times\into k$.

\subsection{Semistable models}
A {\em semistable model} of $(X,D)$ is determined by a formal model $\gtX$ of $X$ such that the reduction map $X\to\gtX_s$ maps $D$ bijectively onto a set $D_s\subset\gtX_s$ of $\tilk$-smooth points. In this case, the closed immersion $D\into X$ extends to a closed immersion $\gtD\into\gtX$ with $\gtD$ a disjoint union of copies of $\Spf(\kcirc)$, and we call $(\gtX,\gtD)$ a semistable formal model of $(X,D)$. This can also be promoted to the logarithmic geometry by setting $M_{\gtX}=M_X\cap\calO_{\gtX}$. Loosely speaking, $M_{\gtX}$ is the subsheaf of $\calO_{\gtX}$ consisting of functions invertible on the complement of $\gtX_s\cup\gtD$. The main advantage of working on the log level is that $\gtX^\rmlog=(\gtX,M_{\gtX})$ is log smooth over $\gtS^\rmlog=(\gtS,M_{\gtS})$, where $\gtS=\Spf(\kcirc)$ and $M_{\gtS}=\kcirc\setminus\{0\}\into\kcirc$.

\subsection{Reduction}
Restricting everything to the closed fiber $s^\rmlog$ we obtain an $s^\rmlog$-smooth enhancement $\tilX^\rmlog=(\tilX,M_\tilX)$ of the closed fiber $\tilX=\gtX_s$. Clearly, $\tilX^\rmlog$ is the log reduction of $X$ in the sense of \S\ref{logredsec}, and the construction is functorial: if $f\:Y\to X$ is as in \S\ref{convsec} and $V_Y,V_X$ are compatible triangulations, then $f$ induces morphisms of the associated log enhanced models $\gtY^\rmlog\to\gtX^\rmlog$ and their closed fibers, yielding a log reduction morphism $\tilY^\rmlog\to\tilX^\rmlog$.

\bibliographystyle{amsalpha}
\bibliography{liftings}

\end{document}